\documentclass{siamltex}

\usepackage{epsfig}

\usepackage{amssymb}
\usepackage{amsfonts}
\usepackage{amsmath}

\newcommand{\re}{\mathrm{e}}
\newcommand{\ri}{\mathrm{i}}
\newcommand{\rd}{\mathrm{d}}

\newcommand\R[0]{{\mathbb R}}
\newcommand\N[0]{{\mathbb N}}
\newcommand\Z[0]{{\mathbb Z}}
\newcommand\C[0]{{\mathbb C}}

\newcommand{\cond}{{\rm cond \;}}
\newcommand{\esssup}{\mathop{{\rm ess} \sup}}
\newcommand{\essinf}{\mathop{{\rm ess} \inf}}

\newcommand{\Ak}{A_{k,\eta}}

\newcommand{\cA}{A}

\newcommand{\Ake}{A_{k,\eta}}

\newcommand{\Apke}{{A^\prime}_{k,\eta}}
\newcommand{\Apk}{\Apke}

\title{Condition number estimates for combined potential integral operators
in acoustics and their boundary element discretisation}

\author{ T Betcke\footnotemark[1]\ \footnotemark[2]\ \footnotemark[7]\and S N Chandler-Wilde\footnotemark[1]\ \footnotemark[3]\ \footnotemark[7]\and I G Graham\footnotemark[4]\ \footnotemark[7]
       \and S Langdon\footnotemark[1]\ \footnotemark[5]\ \footnotemark[7]\and M Lindner\footnotemark[6]\ \footnotemark[7]}

\begin{document}

\maketitle

\renewcommand{\thefootnote}{\fnsymbol{footnote}}
\footnotetext[1]{Department of Mathematics and Statistics, University of Reading,
Whiteknights, PO Box 220, Berkshire RG6 6AX, UK}
\footnotetext[2]{Email: {\tt t.betcke@reading.ac.uk}}
\footnotetext[3]{Email: {\tt s.n.chandler-wilde@reading.ac.uk}}
\footnotetext[4]{Department of Mathematical Sciences,
University of Bath, Bath BA2 7AY, UK. Email: {\tt I.G.Graham@bath.ac.uk}}
\footnotetext[5]{Email: {\tt s.langdon@reading.ac.uk}}
\footnotetext[6]{TU Chemnitz, Fakult\"{a}t f\"{u}r Mathematik, 09107 Chemnitz, Germany. Email: {\tt mali@hrz.tu-chemnitz.de}}
\footnotetext[7]{The authors gratefully acknowledge  helpful discussions with Alex Barnett (Dartmouth), Peter Chamberlain (Reading), and Michael Levitin (Reading). Timo Betcke is supported by Engineering and Physical Sciences Research Council (EPSRC) Grant EP/H004009/1, Marko Lindner is partially supported by a Marie Curie Fellowship of the European Commission (MEIF-CT-2005-009758), and the other authors are supported by EPSRC Grant EP/F067798/1.}

\renewcommand{\thefootnote}{\arabic{footnote}}

\bibliographystyle{siam}

\begin{abstract}
  We consider the classical coupled, combined-field integral equation formulations for time-harmonic acoustic scattering
  by a sound soft bounded obstacle.  In recent work, we have proved
  lower and upper bounds on the
$L^2$ condition numbers
  for these formulations, and also on the norms of the classical acoustic single- and
  double-layer potential operators. These bounds to some extent make explicit the dependence of condition numbers on the wave number $k$, the geometry of the
  scatterer, and the coupling parameter. For example, with the usual choice of coupling parameter they show that, while the condition number grows like $k^{1/3}$ as $k\to\infty$, when the scatterer is a circle or sphere, it can grow as fast as $k^{7/5}$ for a class of `trapping' obstacles. In this paper we prove further bounds, sharpening and extending our previous results. In particular we show that there exist trapping obstacles for which the condition numbers grow as fast as $\exp(\gamma k)$, for some $\gamma>0$, as $k\to\infty$ through some sequence. This result depends on exponential localisation bounds on Laplace eigenfunctions in an ellipse that we prove in the appendix. We also clarify the correct choice of coupling parameter in 2D for low $k$.  In the second part of the paper we focus on the boundary element discretisation of these operators.
  We discuss the extent to which the bounds on the continuous operators are also satisfied by their discrete
  counterparts and, via numerical experiments, we provide supporting evidence for some of the
  theoretical results, both quantitative and asymptotic, indicating further which of the upper and lower bounds may be
  sharper.
\end{abstract}




\section{Introduction}
\label{sec:intro}
\setcounter{equation}{0}

Consider scattering of a time-harmonic ($\re^{-\ri\omega t}$ time dependence)
acoustic wave $u^i$ by a bounded, sound soft obstacle occupying a compact set
$\Omega\subset \R^d$ ($d=2$ or 3) with Lipschitz boundary $\Gamma$, which is such that the complement set $\Omega_e :=\R^d\setminus \Omega$ is connected.
The
medium of propagation, occupying  $\Omega_e$,  is assumed to be
homogeneous, isotropic and at rest.
Under the assumption that $u^i$ is an entire
solution of the Helmholtz (or reduced wave)
equation with {\em wavenumber} $k=\omega/c>0$
(where $c>0$ denotes the speed of sound), we seek the resulting
time-harmonic acoustic pressure field $u$,  satisfying  the Helmholtz equation
\begin{equation}
  \label{eqn:helmholtz}
  \Delta u + k^2 u = 0 \qquad\textrm{in}\qquad \Omega_e \ .
\end{equation}
This is to be solved subject to
the sound soft boundary condition
\begin{equation}
  \label{DBC}
  u = 0 \qquad\textrm{ on }\qquad \Gamma = \partial \Omega_e,
\end{equation}
and the Sommerfeld radiation condition, which requires that
\begin{equation}
  \label{eqn:src}
  \frac{\partial u^s}{\partial r}-\ri k u^s = o(r^{-(d-1)/2})
\end{equation}
as $r:=|x|\to\infty$, uniformly in $\hat x:= x/r$,
where $u^s:= u-u^i$ represents the scattered part of the field
(see e.g.~\cite{CK83}).
This problem has exactly one solution under the constraint that
$u$ and $\nabla u$ be locally square integrable; see e.g.~\cite{mclean}.

In this paper we consider the two standard second kind boundary integral equation reformulations of (\ref{eqn:helmholtz})--(\ref{eqn:src}).
The first is the {\bf indirect formulation}
\begin{equation}
  \label{eq:central ieq}
  \Ak\varphi\ =\ g,
\end{equation}
where
\[
  \Ak \; := \; I + D_k - \ri\eta S_k,
\]
with $\eta\in\R\backslash\{0\}$ the {\em coupling parameter},
$I$ the identity operator and $S_k$ and $D_k$ the {\em single-} and
{\em double-layer potential operators}.  These are defined for $\varphi\in L^2(\Gamma)$ by
\begin{equation}
  \label{eq:S}
  S_k\varphi(x) \ := \ 2\,\int_{\Gamma} \Phi(x,y)\; \varphi(y) \; ds(y), \quad x \in \Gamma,
\end{equation}
and
\begin{equation}
  \label{eq:D}
  D_k\varphi(x) \ := \ 2\,\int_{\Gamma} \frac{\partial \Phi(x,y)}
  {\partial \nu(y)}\; \varphi(y) \; ds(y), \quad x \in \Gamma,
\end{equation}
with $\partial/\partial \nu(y)$ the
derivative in the normal direction, with the unit normal $\nu(y)$
directed into $\Omega_e$, and $\Phi(x,y)$ the standard
free-space fundamental solution of the Helmholtz equation. This is
given by
\begin{equation}
  \label{eq:Phi}
  \Phi(x,y) := \left\{\begin{array}{cc}
                       \frac{\ri}{4}H_0^{(1)}(k|x-y|), & d=2,
                      \\ \\
                       \displaystyle{\frac{\re^{\ri k|x-y|}}{4\pi|x-y|}}, & d=3,
                    \end{array}\right.
\end{equation}
for $x,y\in\R^d$, $x\neq y$, where $H_0^{(1)}$ is the Hankel
function of the first kind of order zero.  Finally, $g:= -2u^i|_\Gamma = 2u^s|_\Gamma$.

The   second formulation is the {\bf direct formulation}
\begin{equation}
    \Apk \frac{\partial u}{\partial \nu} = f,
    \label{eqn:2ndkind}
\end{equation}
where
\[
  \Apk \ := \ I + D_k^\prime - \ri\eta S_k,
\]
with $D_k^\prime$ the integral operator defined, for $\varphi\in L^2(\Gamma)$, by
\[
  D^\prime_k\varphi(x) \ := \ 2\,\int_{\Gamma} \frac{\partial \Phi(x,y)} {\partial \nu(x)}\; \varphi(y) \; ds(y), \quad x \in \Gamma,
\]
and
\[
  f(x):=2\frac{\partial u^i}{\partial \nu}(x) - 2\ri \eta u^i(x), \quad x\in\Gamma.
\]
It is well known (see~\cite{CWL06} for details, in particular regarding how classical results can be adapted to the general Lipschitz case)
that, for $\eta\neq 0$, $\Ak$ and $\Apk$ are invertible as operators on $L^2(\Gamma)$, and that
\[
  \|\Apk\|=\|\Ak\|, \quad \|{\Apk}^{-1}\|=\|\Ak^{-1}\|.
\]
(Throughout the paper $\Vert \cdot \Vert$ denotes the $L^2$ norm on
$\Gamma$.)

A question that has received much recent attention in the literature
(see for example \cite{Ami90,Ami93,BaSa:07,CWGLL07,CWMonk06,DoGrSm:07,Gi:97,Kr:85,KrSp:83})
is that of determining how the
conditioning of the two standard integral equation formulations, (\ref{eq:central ieq}) and
(\ref{eqn:2ndkind}), depends on the wavenumber $k$, on the coupling parameter $\eta$, and on
the shape of $\Gamma$.  Specifically we are interested in upper and lower bounds on the
(identical) condition numbers of $\Ak$ and $\Apk$, given by
$$
  \cond \Apk = \cond \Ak = \|\Ak\|\,\|\Ak^{-1}\|,
$$
and so we are interested in upper and lower bounds on the norms
of $\Ak$ and its inverse, and also on the norms $\|S_k\|$ and $\|D_k\|$.

In our recent paper~\cite{CWGLL07}, we derived estimates which, to some extent, make explicit the dependence of each of these norms on $k$, $\eta$ and $\Gamma$, with an
emphasis on understanding conditioning in the important but difficult
and relatively neglected case where  $k\to\infty$.
For example, with the usual choice of coupling parameter $\eta=k$, while the condition numbers of $A_{k,\eta}$ and $A_{k,\eta}^\prime$ grow like $k^{1/3}$ as $k\to\infty$ when the scatterer is a circle or sphere \cite{DoGrSm:07}, we show in \cite{CWGLL07} that they grow like $k^{1/2}$ for a starlike polygon and  as fast as $k^{7/5}$ for a class of `trapping' obstacles. In this paper we prove further bounds sharpening and clarifying our previous results, in particular studying trapping obstacles in much more detail.
A main focus of the present paper is also the boundary element discretisation of these operators.  Our aims here are threefold:  to provide supporting evidence for some of the theoretical results of~\cite{CWGLL07} and of \S\ref{sec:previous} via numerical experiments; to determine how sharp the quantitative upper and lower bounds on norms of~\cite{CWGLL07} may be, particularly in the cases where there is a significant gap between the two; to determine the extent to which the bounds on the continuous operators are also satisfied by their discrete counterparts.

We begin in~\S\ref{sec:previous} by summarising the estimates at the continuous level derived in~\cite{CWGLL07}, together with previous related results in the literature. Also, in the 2D case, we sharpen the estimates from~\cite{CWGLL07} at low frequencies, and prove that the choice of $\eta$ in \cite{KrSp:83,Kr:85} (based on analysis for  a circular scatterer) guarantees a bounded condition number in the limit $k\to0$ even for general Lipschitz $\Gamma$. But the main novelty of \S\ref{sec:previous} is that we  show that there exist trapping obstacles for which the condition numbers of $A_{k,\eta}$ and its adjoint  grow as fast as $\exp(\gamma k)$, for some $\gamma>0$, as $k\to\infty$ through some sequence. This result depends on exponential localisation bounds on so-called  `bouncing-ball' type \cite{Keller85} Laplace eigenfunctions  in an ellipse. For completeness we provide a self-contained and relatively elementary proof of this exponential localisation in the appendix; for eigenfunction localisation results in much more general settings proved using related but much more technical arguments see \cite{toth}.

In~\S\ref{sec:discrete} we prove results about the relationship between the continuous integral operators and their discrete counterparts, i.e.\ matrices derived from standard Galerkin boundary element method (BEM) discretisations. In~\S\ref{sec:num} we present numerical results showing Galerkin BEM approximations to $\|\Ak\|$, $\|\Ak^{-1}\|$, $\|S_k\|$ and $\|D_k\|$ for a variety of obstacles, each for a range of values of $k$ and $\eta$.  Finally in~\S\ref{sec:conc} we present some conclusions.

The results of the present  paper and
of \cite{CWGLL07} have direct relevance to the numerical performance
of  boundary
integral methods,
since the condition number of the discretization of
\eqref{eq:central ieq} and \eqref{eqn:2ndkind} appears naturally as a
measure of the difficulty of computing numerical  solutions  in practice. Moreover the results in \cite{CWGLL07}, and more particularly our new results on trapping obstacles,
have direct relevance to a recent detailed  $k-$explicit numerical analysis
of $hp$ boundary integral methods for general Helmholtz scattering
problems in \cite{LoMe10}. There it is shown (for example in
\cite[Corollary 3.18]{LoMe10}) that, provided $\Vert A_{k,k}^{-1}\Vert \leq
C k^\beta$ with $C$ and $\beta$ independent of $k$, then an $hp$ refinement strategy in
which $p$ grows logarithmically in $k$  and $h$ decreases like
$k^{-1} \log k$ yields a Galerkin method which is free from
``pollution'' (i.e.\ the error is bounded by the best possible error
in the finite element space, multiplied by a constant independent of $k$). Our analytical and numerical results are suggestive that
$\Vert A_{k,k}^{-1}\Vert \leq
Ck^\beta$ holds for some $k$-independent $C$ and $\beta$ not only for starlike obstacles, as considered previously in \cite{CWMonk06}, but also for certain trapping obstacles. But also we prove in \S\ref{sec:trapping} that there exist Lipschitz obstacles for which the bound $\Vert A_{k,k}^{-1}\Vert \leq
Ck^\beta$ does not hold for any $C$ and $\beta$. We also note that the paper
\cite{melenk} (a
companion paper to \cite{LoMe10}) contains new decompositions of the
combined potential operators $A_{k,\eta}$ and $A'_{k,\eta}$ which are
crucial in the analysis of the $hp$ methods in \cite{LoMe10}. Of key
importance there is the fact that the decomposition involves
certain operators which map into spaces of  functions which are analytic in a
neighbourhood of $\Gamma$. However this analysis is rather different
in flavour (and has different goals) from that of the present paper.

We flag that a related and complementary study of the same boundary integral equation formulations that we consider in this paper has been carried out recently in \cite{BeSp:10}. That paper includes, similarly to our \S\ref{sec:num}, a numerical study relating to a range of geometries of 2D scatterers, but \cite{BeSp:10} has a different focus, namely an investigation, via computation of the numerical range of boundary element discretisations, of conditions which ensure that $A_{k,k}$ is coercive, and how its coercivity constant depends on $k$.

Finally we note that, in a similar vein to our \S\ref{sec:num}, Warnick and Chew \cite{WaChew99,WaChew01,WaChew04} study the conditioning
of boundary element discretisations of the single-layer potential operator $S_k$ via an
approximate theoretical analysis and numerical experiments, obtaining
simple explicit approximate upper and lower bounds for the condition
number as a function of $k$ and the discretisation step size for several
canonical 2D geometries (a circle, crack and two parallel cracks) \cite[Table 2]{WaChew04}.

\section{Bounds on norms and condition numbers at the continuous level}
\label{sec:previous}
\setcounter{equation}{0}


\subsection{The case of a circle or sphere} \label{sec:circ}

Prior to~\cite{CWGLL07}, most research was focussed on the case when $\Gamma$ is a
circle or sphere, in which case Fourier analysis methods are possible.

For the case $d=2$, when $\Gamma$ is the unit circle, rigorous upper bounds
on $\|\Ak\|$ and $\|\Ak^{-1}\|$ for the case $\eta=k$ (previously proposed as
optimal for conditioning for the unit circle when $k\geq 1$ in e.g.\
\cite{Ami90,Ami93,KrSp:83}) were derived in~\cite{DoGrSm:07} and are that,
for all sufficiently large~$k$,
\begin{eqnarray}
  &\Vert \cA_{k,k} \Vert \leq C k^{1/3},&  \label{eq:DoGrSmnorm} \\
  &\|A_{k,k}^{-1}\| \leq 1,& \label{eq:DoGrSmnorminv2}
\end{eqnarray}
with $C$ a constant independent of~$k$. (Combining \eqref{eq:DoGrSmnorminv2} with Lemma \ref{lem} below we see that, in fact, $\|A_{k,k}^{-1}\| = 1$ for all sufficiently large $k$.) Although the focus in \cite{DoGrSm:07} was on bounding $\cA_{k,k}$ rather than on bounding the separate components $S_k$ and $D_k$, Lemmas 4.1, 4.9 and 4.10 in \cite{DoGrSm:07} also imply the separate bounds that
\begin{equation} \label{circ:SDnorm}
\|S_k\| \leq C k^{-2/3}, \quad \|D_k\| \leq C k^{1/3},
\end{equation}
with $C$ a constant independent of $k$.

For the case $d = 3$, when $\Gamma$ is a sphere of unit radius, it is further
shown in~\cite{DoGrSm:07} that, for all sufficiently large $k$, \eqref{eq:DoGrSmnorm} holds
(see also \cite{Gi:97}) and that, for every $C^\prime>1$,
\[
  \Vert \cA_{k,k}^{-1}\Vert \leq C^\prime,
\]
for all sufficiently large $k$.  A more refined and flexible upper bound on $A_{k,\eta}$ than
(\ref{eq:DoGrSmnorm}) in the 3D case was recently derived in~\cite{BaSa:07}, where it was shown that,
for all sufficiently large~$k$,
\begin{equation}
  \|D_k\| \leq C, \quad \|S_k\| \leq Ck^{-2/3},
  \label{eq:DSboundsSphere}
\end{equation}
for some constant $C$ independent of $k$, and hence
\begin{equation}
  \|\Ak\|\ =\ \|I + D_k - \ri\eta S_k\|\ \leq\ 1+ C \left(1 + |\eta| k^{-2/3}\right).
  \label{eq:normAsphere}
\end{equation}
The choice $|\eta| = k$ yields the same estimate as
\eqref{eq:DoGrSmnorm}, whereas the choice $|\eta| = k^{2/3}$ yields
a $k-$independent bound for $\Vert \Ake\Vert$.

\subsection{The case of a starlike obstacle}
\label{sec:starlike}

Consider the case when $\Omega$ is connected, piecewise smooth and
starlike, with $\Gamma$ Lipschitz and $C^2$ in a neighbourhood of almost every $x\in\Gamma$, and
$$
  \delta_- := \essinf_{x\in \Gamma} \; x \cdot \nu(x) > 0
$$
(assuming, without loss of generality, that the origin lies in $\Omega$ ($0\in\Omega$)).
Under these assumptions it is shown in~\cite{CWMonk06} that, for $\eta\in\R\setminus\{0\}$,
\begin{equation}
  \label{eqn:normsequal2}
  \|A_{k,\eta}^{-1}\| \leq B,
\end{equation}
where
$$
  \!\!\!\!\!\!\!\!B:=
  \frac{1}{2}+\left[\left(\frac{\delta_+}{\delta_-}+\frac{4{\delta^*}^2}{\delta_-^2}\right)\left[\frac{\delta_+}{\delta_-}\left(\frac{k^2}{\eta^2}+1\right)
  + \frac{d-2}{\delta_-|\eta|}+\frac{{\delta^*}^2}{\delta_-^2}\right]
  + \frac{(1+2kR_0)^2}{2\delta_-^2\eta^2}\right]^{1/2},
$$
with
\[
  \!R_0 := \sup_{x\in\Gamma}|x|, \quad \delta_+ := \esssup_{x\in \Gamma}
  \; x \cdot \nu(x), \;\; \delta^* := \esssup_{x\in \Gamma} |x-(x
  \cdot \nu(x))\nu(x)|.
\]
These assumptions hold, for example, if
$\Omega$ is a starlike polygon or polyhedron (and $0\in\Omega$), and in these cases $\delta_-$ and $\delta_+$ are
the distances from the origin to the nearest and furthest sides of $\Gamma$, respectively. Note that the
expression $B$ blows up if $k/|\eta|\to\infty$ or if
$\delta_+/\delta_-\to\infty$, or if $\delta_-|\eta|\to 0$, uniformly
with respect to the values of other variables.
If $\Gamma$ is a circle or sphere, i.e.\ $\Gamma=\{x:|x|=R_0\}$, then
$\delta_-=\delta_+=R_0$ and $\delta^*=0$ so
\begin{equation}
  \label{eqn:Bcircle}
  B = B_0 := \frac{1}{2}+\left[1+ \frac{k^2}{\eta^2} + \frac{d-2}{R_0|\eta|} + \frac{(1+2kR_0)^2}{2R_0^2\eta^2}\right]^{1/2}.
\end{equation}
In the general case, since $\delta_-\leq \delta_+\leq R_0$ and
$0\leq \delta_*\leq R_0$, it holds that $B\geq B_0$.

Based on low frequency asymptotics and numerical calculations for the case when $\Gamma$ is a circle, it is proposed in \cite{Kr:85} to choose
\begin{equation} \label{eta_opt}
\eta = \max\left(\frac{1}{2R_0},k\right)
\end{equation}
to minimise the condition number of $A_{k,\eta}$ (and see \cite{Ami90,Ami93} for some further evidence supporting this choice). Based on computational experience, Bruno and Kunyansky \cite{BrKu:01,BrPC07} recommend the
similar formula that $\eta = \max(6T^{-1},k/\pi)$,  where $T$ is the diameter of $\Omega$, on the basis that this choice is found to minimise the number of GMRES
iterations in an iterative solver. With either of these choices $\|A_{k,\eta}^{-1}\|$ is bounded uniformly in $k$ for $k>0$ for $\Omega $ starlike. In particular, with the choice \eqref{eta_opt} we see that
\begin{eqnarray} \nonumber
\|A_{k,\eta}^{-1}\| \leq B & \leq &
  \frac{1}{2}+\left[\left(\frac{\delta_+}{\delta_-}+\frac{4{\delta^*}^2}{\delta_-^2}\right)\left[2\frac{\delta_+}{\delta_-}
  + \frac{2(d-2)R_0}{\delta_-}+\frac{{\delta^*}^2}{\delta_-^2}\right]
  + \frac{8R_0^2}{\delta_-^2}\right]^{1/2}\\
  & \leq & \frac{1}{2}+\theta\left[4+13\theta+4\theta^2\right]^{1/2},\label{Betak}
\end{eqnarray}
where $\theta = R_0/\delta_-$.

\subsection{Upper bounds on $\|S_k\|$, $\|D_k\|$ and $\|\Ak\|$ in the general Lipschitz case} \label{sec:upper}

It is shown in~\cite[Theorems~3.3, 3.5, 3.6]{CWGLL07}, under the assumption
that the scatterer $\Omega$ is Lipschitz, that there exist positive constants
$C_i$, $i=1,2,3$, dependent only on $\Omega$, such that
\begin{eqnarray}
  &\|S_k\| \leq C_1 k^{(d-3)/2},& \label{eq:Sk2Dupper} \\
  &\|D_k\| \leq C_2 k^{(d-1)/2} + C_3,& \label{eq:Dk2Dupper} \\
  &\|\Ak\| \leq 1 + C_3 + C_2 k^{(d-1)/2} + C_1 |\eta| k^{(d-3)/2},& \label{Marko_fb}
\end{eqnarray}
for $k>0$.  In 2D ($d=2$), for the case
$\Gamma$ simply-connected and smooth, (\ref{Marko_fb}) was shown
previously, for all sufficiently large $k$, in \cite{DoGrSm:07}.

Expressions that are in principle computable for the constants $C_i$, $i=1,2,3$, are given in~\cite{CWGLL07}.  In particular,
in the simplest case that $\Gamma$ is a straight line of length $a$, the upper bound on $\|S_k\|$ is given explicitly by
\begin{equation}
  \|S_k\| \leq 2\sqrt{\frac{a}{\pi k}}.
  \label{eqn:Skcrack}
\end{equation}

The bounds \eqref{eq:Sk2Dupper}--\eqref{Marko_fb} are sharp in their dependence on $k$ in the limit $k\rightarrow0$ except for~(\ref{eq:Sk2Dupper}) (and so~(\ref{Marko_fb})) in the 2D case. In the 2D case
the low frequency behaviour is more subtle, as studied previously for the case of a circle in~\cite{KrSp:83, Kr:85}.  To obtain sharper bounds for low $k$ for general Lipschitz $\Gamma$ in the 2D case, note that, from the power series representations for $Y_0$ \cite[(9.1.13)]{AbSt:72}, it follows easily that
\[ \left|Y_0(t)-\frac{2}{\pi}\log t \right| \leq \frac{2}{\pi}\left\{\log 2 - \gamma + \frac{1}{4}\right\}, \]
for $0<t\leq 1$, where $\gamma=0.577\ldots$ is Euler's constant.  Since also $|J_0(t)|\leq 1$ for $t\geq 0$ \cite[(9.1.18)]{AbSt:72}, this implies that
\begin{equation}
  \left| H_0^{(1)}(t)-\frac{2\ri}{\pi}\log t\right| \leq \sqrt{1+\frac{4}{\pi^2}\left(\log 2 - \gamma +\frac{1}{4}\right)^2} < 1.03, \quad 0<t \leq 1.
  \label{eqn:H0bound}
\end{equation}
(Since $H_0^{(1)}(t)=-\frac{2\ri}{\pi}\log t+1+o(1)$ as $t\rightarrow 0^+$, this upper bound is an overestimate by not more than 3\% for small $t$.)

Let $S_0$ denote the single-layer potential operator in the Laplace case, defined by~(\ref{eq:S}) with $\Phi(x,y)$ replaced by $\Phi_0(x,y):=(1/2\pi)\log(R_0/|x-y|)$, for some constant $R_0>0$ (later we will choose $R_0$ to be some characteristic length scale of $\Gamma$).  It is a known result (e.g.~\cite{mclean}) that $S_0$ is a bounded operator on $L^2(\Gamma)$.  Further, (\ref{eqn:H0bound}) implies that, where $D:=\sup_{x,y\in \Gamma}|x-y|$ is the diameter of $\Gamma$,
\begin{equation} \label{phidif}
 \left|\Phi(x,y) - \Phi_0(x,y)+ \frac{1}{2\pi}\log(kR_0)\right| < 0.26,
 \end{equation}
for $kD\leq 1$. From this inequality it follows that
\[
  \|S_k\| \leq \|S_0\|+\| S_k - S_0 \| \leq \|S_0\|+\left( \frac{|\log kR_0|}{2\pi} + 0.26 \right)|\Gamma|,
\]
for $kD\leq 1$, where $|\Gamma|=\int_{\Gamma}\rd s$ is the length of $\Gamma$.  Thus, taking $R_0=1$ in the above result, and combining this bound with (\ref{eq:Sk2Dupper}), we obtain a refined version of~(\ref{eq:Sk2Dupper}) for small $k$ when $d=2$, that
\[
  \|S_k\| \leq C_0 (1-\log k), \quad \mbox{for }0<k\leq 1,
\]
where the positive constant $C_0$ again depends only on $\Omega$,
which, combined with~(\ref{eq:Dk2Dupper}), gives that
\begin{equation} \label{akb}
\|A_{k,\eta}\| \leq 1 + C_3 + C_2 k^{1/2} + C_0|\eta|(1-\log k), \quad 0<k\leq 1.
\end{equation}

\subsection{Lower bounds on $\|S_k\|$, $\|D_k\|$ and $\|\Ak\|$} \label{sec:lower}

The following lower bounds on $\|S_k\|$, $\|D_k\|$ and $\|\Ak\|$ are derived in~\cite[\S4]{CWGLL07}.

\begin{lemma}\cite[Lemma~4.1]{CWGLL07}
  \label{lem}
  In both 2D and 3D, if a part of $\Gamma$ is $C^1$, then $\|\Ak\|\geq 1$, $\|\Ak^{-1}\|\geq 1$.
\end{lemma}

\begin{theorem}\cite[Theorem~4.2]{CWGLL07}
  \label{th:3}
  In the 2D case, if $\Gamma$ contains a straight line section of
  length $a$, then
  $$
    \|S_k\| \geq \sqrt{\frac{a}{\pi k}} + O(k^{-1})
  $$
  as $k\to\infty$ and
  $$
    \|\Ak\| \geq |\eta|\,\sqrt{\frac{a}{\pi k}} - 1 + O(|\eta|k^{-1})
  $$
  as $k\to\infty$, uniformly in $\eta$.
\end{theorem}

\begin{theorem}\cite[Theorem~4.4]{CWGLL07}
  \label{th:5}
  In the 2D case, if $\Gamma$ is locally $C^2$ in
  a neighbourhood of some point $x^*$ on the boundary then, for some
  constants $C>0$ and $k_0>0$, it holds for all $k\geq k_0$ and all
  $\eta\in\R$ that
  $$
    \|S_k\|\geq C k^{-2/3} \;\;\mbox{ and }\;\; \|\Ak\|\geq C|\eta|k^{-2/3}.
  $$

  More generally, adopt a local coordinate system $OX_1X_2$ with
  origin at $x^*$ and the $X_1$ axis in the tangential direction at
  $x^*$, so that, near $x^*$, $\Gamma$ coincides with the curve $\{x^*
  + t^* X_1 + n^*f(X_1):X_1\in \R\}$, for some $f\in C^2(\R)$ with
  $f(0)=f^\prime(0)=0$; here $t^*$ and $n^*$ are the unit tangent and
  normal vectors at $x^*$. Then if, for some $N\in\N$, $\Gamma$ is
  locally $C^{N+1}$ near $x^*$, i.e.\ $f\in C^{N+1}(\R)$, and if also
  $f^{\prime}(0)=f^{(2)}(0)=\dots = f^{(N)}(0)=0$, then there exist
  $C>0$ and $k_0>0$ such that
  $$
    \|S_k\|\geq C k^{-(N+1)/(2N+1)}\;\;\mbox{ and }\;\; \|\Ak\|\geq
    C|\eta| k^{-(N+1)/(2N+1)}
  $$
  for all $k\geq k_0$ and all $\eta\in\R$.
\end{theorem}

In fact, under the conditions of Theorem~\ref{th:5}, assuming further that $f^{(N+1)}(0)\neq 0$, we have quantitative lower bounds on $\|S_k\|$ and $\|A_{k,\eta}\|$:
\[
  \|S_k\|   \geq  C_N(0)\, k^{-(N+1)/(2N+1)}\,(1+o(1)), \quad \mbox{as }k\to\infty,
\]
and
\[
  \|\Ak\| \geq \left\{ \begin{array}{ll}|\eta| \, C_N(0)\, k^{-(N+1)/(2N+1)}\,(1+o(1)), & \mbox{if }|\eta| k^{-(N+1)/(2N+1)}\to\infty, \\
                                        |\eta| \, C_N(0)\, k^{-(N+1)/(2N+1)} - \frac{N}{2\sqrt{2}} +o(1), & \mbox{if }|\eta| \approx k^{(N+1)/(2N+1)},
                                        \end{array}\right.
\]
as $k\to\infty$, where
$$
  C_N(0) = \sqrt{\frac{1}{8\pi}}\left(\sqrt{\frac{\pi}{2}}\frac{N!}{|f^{(N+1)}(0)|}\right)^{1/(2N+1)}.
$$

Noting that $f^{\prime\prime}(0)$ is the
curvature at $x^*$, we have the following corollary by
applying these equations with $N=1$.

\begin{corollary}\cite[Corollary~4.5]{CWGLL07} \label{cor45}
Suppose (in the 2D case) that $\Gamma$ is locally $C^2$ in a
neighbourhood of some point $x^*$ on the boundary and let $R$ be the
radius of curvature at $x^*$. If $R<\infty$,
then,
\begin{equation} \label{eq:Sklower}
\|S_k\| \geq \frac{1}{2}\left(\frac{R}{\pi}\right)^{1/3} \,
(2k)^{-2/3}(1+o(1)), \quad \mbox{as }k\to\infty,
\end{equation}
and
\[
  \|\Ak\| \geq \left\{ \begin{array}{ll} \frac{|\eta|}{2}\left(\frac{R}{\pi}\right)^{1/3} \, (2k)^{-2/3}(1+o(1)), & \mbox{if }|\eta|k^{-2/3}\to\infty, \\
                                        \frac{|\eta|}{2}\left(\frac{R}{\pi}\right)^{1/3} \, (2k)^{-2/3} - \frac{1}{2\sqrt{2}} + o(1), & \mbox{if }|\eta|\approx k^{2/3},
                                        \end{array}\right.
\]
as $k\to\infty$.
\end{corollary}

We also have the following lower bounds on $\|D_k\|$.  The conditions of
Theorem~\ref{th:6} are satisfied, for example, if $\Gamma$ is a polygon. (Choose $x^1$ to
be a corner of the polygon and $x^2$ to be some point on an adjacent
side, with $\Gamma^1$ a neighbourhood of $x^1$ on the adjacent side to $x^2$ and $\Gamma^2$.)

\begin{theorem}\cite[Theorem~4.6]{CWGLL07} \label{th:6}
In the 2D case, suppose $x^1$ and $x^2$ are distinct points on
$\Gamma$, that $\Gamma$ is $C^1$ in one-sided neighbourhoods
$\Gamma^1$ and $\Gamma^2$ of $x^1$ and $x^2$, and that
$(x^1-x^2)\cdot\nu(x)=0$ for $x\in\Gamma^2$ while $(x^1-x^2)$ is not
parallel to $\Gamma^1$ at $x^1$. Then, for some constants $C>0$ and
$k_0>0$, it holds for all $k\ge k_0$ that $\|D_k\|\ge Ck^{1/4}$.
\end{theorem}

The conditions of the next theorem are satisfied
with $N=0$ by some pair of points $x^1$ and $x^2$ whenever $\Gamma$
is $C^1$.

\begin{theorem}\cite[Theorem~4.7]{CWGLL07} \label{th:7}
In the 2D case, suppose $x^1$ and $x^2$ are distinct points on
$\Gamma$, and that, for some $N\in\N_0:= \N\cup\{0\}$, $\Gamma$ is $C^1$ and
$C^{N+1}$ in one-sided neighbourhoods $\Gamma^1$ and $\Gamma^2$ of
$x^1$ and $x^2$, respectively, and that $x^1-x^2$ is not parallel to
$\Gamma^1$ at $x^1$. Without loss of generality, choose $\Gamma^2$
so that, for some $\tilde \epsilon>0$ and $f\in C^{N+1}(\R)$ with
$f(0)=0$,
\[
\Gamma^2\ =\ \{x^2+t\hat u+f(t)\hat n\ :\ 0\le t\le\tilde\epsilon\}
\]
where $\hat u=(x^2-x^1)/|x^1-x^2|$ and $\hat n$ are orthogonal unit
vectors, and suppose that, for some $N\in\N_0$,
\[
f^{(0)}(0)\ =\ f^{(1)}(0)\ =\ \cdots\ =\ f^{(N)}(0)\ =\ 0.
\]
Then there exist $C>0$ and $k_0>0$ such that
\[
\|D_k\|\ \ge\ Ck^{N/(4N+4)}
\]
for all $k>k_0$.
\end{theorem}

\subsection{Lower bounds on $\Vert A_{k,\eta}^{-1}\Vert $ for trapping obstacles}
\label{sec:trapping}
In \cite{CWGLL07} it is shown that if $\Omega$ is a certain type of trapping obstacle then $\Vert A_{k,\eta}^{-1}\Vert $ can be unbounded as $k\to\infty$. The type of trapping obstacle considered in \cite{CWGLL07} is an obstacle for which there exists points $P$ and $Q$ on the boundary $\Gamma$ such that:

(i) $\Gamma$ is $C^1$ in neighbourhoods of $P$ and $Q$;

(ii) the line segment joining $P$ and $Q$ lies in $\Omega_e$ and;

(iii) this line segment is normal to $\Gamma$ at $P$ and $Q$. \\
The line segment $PQ$ is an example of a {\em periodic orbit}, by which we mean that it is the possible locus of a point billiard particle moving in the exterior region $\Omega_e$ in a straight line at unit speed as on an ideal billiard table, interacting with the boundary $\Gamma$ according to the usual law of specular reflection (angle of reflection equals angle of incidence).

The specific class of trapping obstacle discussed in \cite{CWGLL07} is one for which $\Gamma$ is a straight line locally to both  $P$ and $Q$. Precisely, the following theorem is proved,  showing that $\Vert A_{k,\eta}^{-1}\Vert $ is unbounded as $k \rightarrow \infty$ for some class of trapping obstacles, at least provided $\vert \eta\vert \leq Ck$ for
some constant $C$, which is the case of course for the standard choice $\eta=k$.

\begin{theorem} \cite[Theorem~5.1]{CWGLL07}\label{th:trapping}
There exists $C>0$ such that, if $\Omega_e$ contains a square of
side length $2a$, two parallel sides of which form part of $\Gamma$,
 and
$\eta\in\R\setminus\{0\}$, then
\[
\Vert A_{k_m ,\eta}^{-1} \Vert  \ \geq \ C \  k_m^{9/10} \left( 1 +
\frac{\vert\eta \vert}{k_m} \right)^{-1}, \quad m\in\N,
\]
where $k_m:= m \pi /2a$.
\end{theorem}

Theorem \ref{th:trapping} relates to the case when the periodic orbit is between straight line parts of $\Gamma$. A key idea in its proof is the construction of a  {\em quasimode} for the Helmholtz equation in  $\Omega_e$, by which we mean a function $v\in H^2(\Omega_e)$ which satisfies $\Delta v + k^2 v = g$ with $g\in L^2(\Omega_e)$ having a small norm relative to that of $\Delta v$; precisely, the quasimode is constructed, dependent on $k$, in such a way that if $k=k_m = m\pi/(2a)$ then $\|g\|_{L^2(\Omega_e)}/\|\Delta v\|_{L^2(\Omega_e)} = O(k^{-2})$ as $k\to\infty$. We note that the rate of growth of $\|A_{k,\eta}^{-1}\|$ predicted in Theorem \ref{th:trapping} will be confirmed by numerical calculations in \S4, and cf.\ \cite[Fig.~4.7] {LoMe10}.

A periodic orbit between two parallel straight lines is neutrally stable, by which we mean that a small initial perturbation in the point billiard's position or direction will cause a perturbation to the billiard motion which grows at most linearly with time. If the parts of $\Gamma$ neighbouring $P$ and $Q$ are curved slightly, so that the periodic orbit becomes stable, then the construction of a quasimode becomes possible for which $\|g\|_{L^2(\Omega_e)}/\|\Delta v\|_{L^2(\Omega_e)}$ decreases very rapidly as $k\to\infty$ through some unbounded sequence of values (see \cite{Keller85,Lazutkin99} and the references therein), which leads to a very fast growth in  $\|A_{k ,\eta}^{-1} \|$ as $k\to\infty$ through the same sequence of values.

We will prove this statement in Theorem \ref{th:trapping2} below in a case for which a complete proof can be given by fairly elementary arguments. This simplest case is that in which the parts of $\Gamma$ neighbouring $P$ and $Q$ form part of the boundary of an ellipse, precisely an ellipse of which $PQ$ is the shortest periodic orbit, in which case the quasimode can be constructed by perturbing a so-called  {\em bouncing ball mode} (see \cite{Keller85}) eigenfunction of the ellipse. This mode can be written down explicitly in terms of Mathieu functions and can be shown to be exponentially localised around the stable periodic orbit $PQ$. (For details  see the appendix, and for a visualization of several of these eigenfunctions see Figure \ref{fig:ellipsemodes} below.) An example of an exterior domain $\Omega_e$ and the corresponding scattering object $\Omega$ which satisfies the conditions of Theorem \ref{th:trapping2} is the obstacle labeled `Elliptic cavity' in Figure \ref{fig:shapes1} below.

\begin{theorem} \label{th:trapping2}
If, for some $a_1>a_2>0$, $\Omega_e$ contains the ellipse $E := \{(x_1,x_2):(x_1/a_1)^2+(x_2/a_2)^2 < 1\}$, and if $\Gamma$ coincides with the boundary of this ellipse in neighbourhoods of the points $(0,\pm a_2)$, then there exists a sequence $0< k_0< k_1<k_2<\ldots$, with $k_m\to\infty$ as $m\to\infty$, such that, for some $\gamma >0$ and $C>0$,
\begin{equation}\label{eq:tr22}
\Vert A_{k_m ,\eta}^{-1} \Vert  \geq C \re^{\gamma k_m}\, \left(1+ \frac{\vert\eta \vert}{k_m} \right)^{-1},
\end{equation}
for
$\eta\in\R\setminus\{0\}$ and $m = 0,1,2,\ldots\,$.
\end{theorem}
\begin{proof}
 In the appendix we focus on a particular subset of the eigenfunctions of the Laplace operator with Dirichlet boundary conditions in the ellipse $E$. These are the functions $u_{m,0}\in C^2(\bar E)$, $m=0,1,\ldots$, defined by $u_{m,0}(x) =$ $\mathrm{Mc}_0^{(1)}(\mu,q_m) \mathrm{ce}_0(\nu, q_m)$, $x\in E$, where the elliptic coordinates $(\mu,\nu)$ and standard Mathieu function notation are as defined in the appendix. The important property of the function  $u_{m,0}$  is that it satisfies the eigenproblem \eqref{eigen} for wavenumber $k=k_m$, where $k_m = 2\sqrt{q_m/(a_1^2-a_2^2)}$ and $q_m$ is the $(m+1)$th positive solution of the equation \eqref{mode} in the case $n=0$, with $q_m\to \infty$ (so that $k_m\to\infty$) as $m\to\infty$.
 It is shown in the appendix (see \eqref{finalfinal}) that this particular subset of eigenfunctions $u_{m,0}$, $m=0,1,\ldots$,  is a family of bouncing ball modes, with $u_{m,0}$ becoming increasingly localized around the periodic orbit $ O:= \{(0,x_2): |x_2|\leq a_2\}$ as $m\to\infty$.

 We will now construct a quasimode $v_m$ on $\Omega_e$ by a suitable modification and extension of $u_{m,0}$. Let $\chi\in C^\infty(\R^2)$ be compactly supported and such that $\chi(x)=1$ in some neighbourhood of $O$ while $\chi(x)=0$ in some neighbourhood of $\partial E \setminus \Gamma$. Abbreviate $u_{m,0}$ as $u_m$ and define $v_m\in C^2(\bar \Omega_e)$ by $v_m(x) := \chi(x)u_m(x)$, $x\in \bar E$, $v_m(x) := 0$, $x\in \Omega_e \setminus E$. Then, in $\Omega_e$,
\[
\Delta v_m + k_m^2v_m = g_m,
\]
where $g_m(x) = 0$ for $x\in \Omega_e\setminus E$, while
$$
g_m = u_{m} \Delta \chi   + 2 \nabla \chi \cdot \nabla u_{m}
$$
in $E$. Let $E_- := \{x\in E: |x_1| > \epsilon_-\}$ and $E_+ := \{x\in E: |x_1| > \epsilon_+\}$, where $\epsilon_->\epsilon_+ >0$ are chosen sufficiently small so that $\chi = 1$ in $E\setminus E_-$. Then, where $\|.\|_\infty$ denotes the usual supremum norm on $C(\bar E)$ and $\|\cdot\|_2$ the usual norm on $L^2(E)$ and $\|\cdot\|_{L^2(E_\pm)}$ the $L^2$ norm on $E_\pm$, we see that
$$
\|g_m\|_2 \leq \|u_m\|_{L^2(E_-)} \|\chi\|_\infty + 2 \|\nabla \chi\|_\infty \|\nabla u_m\|_{L^2(E_-)}.
$$

In the remainder of the proof let $C$ denote a positive constant, whose value does not depend on $m$, but which is not necessarily the same at each occurrence. By \eqref{finalfinal}, for some $\beta>0$,
\begin{equation} \label{bound3}
\|u_m\|_{L^2(E_\pm)} \leq C \re^{-\beta k_m} \|u_m\|_2,
\end{equation}
for $m=0,1,\ldots\,$.
Further, $\|\nabla u_m\|_{L^2(E_-)}$ can be bounded by a constant multiple of $k_m\|u_m\|_{L^2(E_+)}$, so that
\begin{equation} \label{boundb}
\|g_m\|_2 \leq C k_m \re^{-\beta k_m} \|u_m\|_2, \quad m = 0,1, \ldots \, .
\end{equation}

 To see this last claim choose $f\in C^1(\R)$ such that: (i) $1\geq f(s) \geq 0$ for $s\in \R$; (ii) $f(s) = 0$ for $|s|\leq \epsilon_+$; (iii) $f(s) = 1$ for $|s| \geq \epsilon_-$;  (iv) for some constant $M>0$, $|f^\prime(s)|/(f(s))^{1/2} \leq M$ for all $s\in \R$ for which $f(s)> 0$. (This can be achieved by defining $f$ by $f(s) = P((|s|-\epsilon_+)/(\epsilon_--\epsilon_+))$ for $\epsilon_+\leq |s| \leq \epsilon_-$, where $P(t) = t^2(3-2t)$.) Define $\tilde \chi \in C^1(\bar E)$ by $\tilde \chi(x) = f(x_1)$, $x\in \bar E$, and note that $1\geq \tilde \chi(x)\geq 0$, for $x\in E$, that $\tilde \chi(x) = 0$ for $x\in E\setminus E_+$ and $\tilde \chi(x) = 1$ for $x\in E_-$, and that $|\nabla \tilde \chi(x)|/\sqrt{\tilde \chi(x)}\leq M$ for all $x\in E$ for which $\tilde \chi(x)>0$. Now, by Green's theorem and since $u_m$ is an eigenfunction in $E$,
 $$
 k_m^2 \int_E \tilde \chi u_m^2 dx = -\int_E \tilde \chi u_m \Delta u_m \, dx = \int_E \nabla(\tilde \chi u_m)\cdot  \nabla u_m \, dx,
 $$
 so that
 $$
 \int_E \tilde \chi (\nabla u_m)^2 dx \leq k_m^2 \int_{E_+}  u_m^2 dx + M \int_E \sqrt{\tilde \chi}\, |\nabla u_m| \, u_m \,dx.
 $$
 Applying Cauchy-Schwarz and noting that $2ab \leq \eta a^2 + \eta^{-1} b^2$, for all $\eta >0$ and $a,b\geq 0$, we see that
 $$
 \|\sqrt{\tilde \chi} \,\nabla u_m \|_2^2 \leq k_m^2 \|u_m\|_{L^2(E_+)}^2 + \frac{M\eta}{2}\|\sqrt{\tilde \chi} \,\nabla u_m \|_2^2 +\frac{M}{2\eta}\|u_m \|_{L^2(E_+)}^2,
 $$
 for all $\eta>0$. Choosing $\eta = M^{-1}$ we see that
 $$
 \|\nabla u_m \|_{L^2(E_-)} \leq \|\sqrt{\tilde \chi} \,\nabla u_m \|_2 \leq c_m \|u_m \|_{L^2(E_+)},
 $$
 where $c_m := \sqrt{M^2 + 2 k_m^2}$.

 Next note that
 \begin{equation} \label{bb}
 \|v_m\|_2^2 \geq \int_{E\setminus E_+} u_m^2 \, dx = \|u_m\|_2^2 - \|u_m\|^2_{L^2(E_+)} \geq \frac{1}{4} \|u_m\|_2^2,
 \end{equation}
 for all sufficiently large $m$, by \eqref{bound3}. Thus, and noting \eqref{boundb},
 $$
 \|\Delta v_m\|_2 = \|k_m^2 v_m - g_m\|_2 \geq k_m^2 \|v_m\|_2 - \|g_m\|_2 \geq \frac{k_m^2}{4} \|u_m\|_2,
 $$
 for all sufficiently large $m$. Combining this bound with \eqref{boundb} we see that
 \begin{equation} \label{firstb}
\frac{\|g_m\|_{L^2(\Omega_e)}}{\|\Delta v_m\|_{L^2(\Omega_e)}} =\frac{\|g_m\|_2}{\|\Delta v_m\|_2} \leq C k_m^{-1} \re^{-\beta k_m} \leq C \re^{-\gamma k_m}, \; m = 0,1,\ldots,
 \end{equation}
 for some $0<\gamma<\beta$.

To see that \eqref{firstb} induces exponential growth of $\Vert A_{k_m ,\eta}^{-1} \Vert$, we proceed as in the proof of \cite[Theorem 5.1]{CWGLL07} and define $v^i_m\in C(\R^2)\cap H_{\mathrm{loc}}^2(\R^2)$ by
\begin{equation} \label{vimdef}
v^i_m(x) := \int_E \Phi_{k_m}(x,y) g_m(y) \, dy, \quad x\in \R^2,
\end{equation}
where $\Phi_{k_m}$ denotes the fundamental solution $\Phi$ of the Helmholtz equation in 2D in the case $k=k_m$. Then we can view $v_m\in C^2(\bar \Omega_e)$ as the total field for the problem of scattering by the obstacle $\Omega$ in the case when $v_m^i$ is the incident field. For defining $v^s_m:= v_m-v_m^i$ it holds that $\Delta v^s_m + k_m^2 v^s_m = 0$ in $\Omega_e$, that $v^s_m$ satisfies the Sommerfeld radiation condition (since $v^i_m$ does and $v_m$ is compactly supported), and that $v^s_m = - v^i_m$ on $\Gamma$. It follows, arguing as in the proof of \cite[Theorem~5.1]{CWGLL07}, that
$$
A^\prime_{k_m ,\eta} \frac{\partial v_m}{\partial \nu} = f_m
$$
(cf.\ \eqref{eqn:2ndkind}), where
\begin{equation} \label{fm}
f_m(x) :=2\frac{\partial v^i_m}{\partial \nu}(x) - 2\ri \eta v^i_m(x), \quad x\in\Gamma.
\end{equation}
Since $\Vert A_{k_m ,\eta}^{-1} \Vert=\Vert (A^\prime_{k_m ,\eta})^{-1} \Vert$, our proof of \eqref{eq:tr22} will be completed if we can show that, for some constant $\gamma>0$,
\begin{equation} \label{boundfinal}
\left\|\frac{\partial v_m}{\partial \nu}\right\|_{L^2(\Gamma)} \geq C \re^{\gamma k_m} \left(1+ \frac{\vert\eta \vert}{k_m} \right)^{-1} \|f_m\|_{L^2(\Gamma)},
\end{equation}
for $m=0,1,\ldots$ and $\eta\in \R\setminus\{0\}$.

To see that \eqref{firstb} implies \eqref{boundfinal}, we use \eqref{vimdef}, and we also  apply Green's representation theorem \cite{CK92} to $v_m$ to give that
\begin{equation} \label{rep3}
v_m(x) = \int_{\Omega_e} \Phi_{k_m}(x,y) g_m(y)\, dy + \int_{\Gamma}\Phi_{k_m}(x,y)
  \frac{\partial v_m}
  {\partial \nu}(y) \, ds(y), \quad x \in \Omega_e.
\end{equation}
Using the bound (e.g.\ \cite{CWGLL07}) that $|H_0^{(1)}(t)|\leq \sqrt{2/(\pi t)}$, for $t > 0$, which implies that
$$
|\Phi_{k_m}(x,y)| \leq (8\pi k_m|x-y|)^{-1/2},
$$
we easily deduce from \eqref{rep3} that
$$
\|v_m\|_2 \leq  C k_m^{-1/2}\left( \|g_m\|_2 + \left\|\frac{\partial v_m}{\partial \nu}\right\|_{L^2(\Gamma)}\right),
$$
for $m=0,1,\ldots\,$.
Combining this bound with \eqref{boundb} and \eqref{bb} we see that
\begin{equation} \label{bb2}
\left\|\frac{\partial v_m}{\partial \nu}\right\|_{L^2(\Gamma)} \geq C k_m^{1/2}  \|u_m\|_2,
\end{equation}
for $m=0,1,\ldots\,$. Similarly, it follows from \eqref{vimdef} that
\begin{equation} \label{bb3}
\|v^i_m\|_{L^2(\Gamma)} \leq C k_m^{-1/2} \|g_m\|_2,
\end{equation}
and that
\[
\nabla v^i_m(x) = w_m^{(0)}(x) + w_m^{(1)}(x),
\]
where, for $x\in \R^2$,
$$
w_m^{(0)}(x) := \int_E \nabla_x \Phi_{0}(x,y) g_m(y) \, dy, \; w_m^{(1)}(x) := \int_E \nabla_x \left(\Phi_{k_m}(x,y)-\Phi_{0}(x,y)\right)
 g_m(y) \, dy,
$$
 and (cf.\ \S2.3) $\Phi_0(x,y):=(1/2\pi)\log(1/|x-y|)$ is the standard fundamental solution of the Laplace equation. Now, from standard mapping properties of Newtonian potentials, it holds that $w_m^{(0)}\in H^1(E)$, with $\|w_m^{(0)}\|_{H^1(E)} \leq C \|g_m\|_2$. Hence, by the boundedness of the standard trace operator from $H^1(E)$ to $H^{1/2}(\partial E)\supset L^2(\partial E)$, it follows that $\|w_m^{(0)}\|_{L^2(\partial E)} \leq C \|g_m\|_2$. Further it holds (see e.g.\ \cite[equation (3.9)]{CWGLL07}) that
$$
\left|\nabla_x \left(\Phi_0(x,y)-\Phi_{k_m}(x,y)\right)\right| \leq C \sqrt{\frac{k_m}{|x-y|}},
$$
from which it follows (cf.\ \eqref{bb3}) that
$$
\|w_m^{(1)}\|_{L^2(\partial E)} \leq C k_m^{1/2}\|g_m\|_2.
$$
Hence
$$
 \left\|\frac{\partial v^i_m}{\partial \nu}\right\|_{L^2(\Gamma)} \leq C k_m^{1/2} \|g_m\|_2,
$$
and combining this bound with \eqref{bb3} and the definition \eqref{fm} of $f_m$, we see that
\begin{equation} \nonumber
\|f_m\|_{L^2(\Gamma)}\leq C k_m^{1/2} \, \left(1+ \frac{\vert\eta \vert}{k_m} \right)\,\|g_m\|_2.
\end{equation}
Finally, combining this bound with \eqref{boundb} and \eqref{bb2}, we see that
\[
\|f_m\|_{L^2(\Gamma)} \leq C k_m \re^{-\beta k_m} \left(1+ \frac{\vert\eta \vert}{k_m} \right) \left\|\frac{\partial v_m}{\partial \nu}\right\|_{L^2(\Gamma)}  ,
\]
which implies that \eqref{boundfinal} holds for $\gamma < \beta$.
\end{proof}

\subsection{Choice of $\eta$ for low $k$} \label{sec:lowk} Although the main focus of this paper is on conditioning in the limit as $k\to\infty$, for completeness we briefly address the limit $k\to0$ in this section. Conditioning in this limit was explored carefully already in the papers \cite{KrSp:83,Kr:85} where, for the case when $\Gamma$ is a sphere or circle, precise asymptotic calculations were made of the choice of $\eta$ which minimises  $\cond A_{k,\eta}$ in the limit $k\to 0$. The recommendations in these papers are for a circle/sphere of unit radius, and imply for a circle/sphere of radius $R_0>0$ that the optimal choices of $\eta$ are
\begin{equation} \label{etachoice}
\eta = \left\{\begin{array}{cc}          \frac{1}{2R_0} + O(k^2 \log k), & d = 3,\\
\{\pi^2 + 4(\log(k/2) + \gamma)^2\}^{-1/2}\{1+O(k^2\log k R_0)\}, & d=2,\end{array}\right.
\end{equation}
where $\gamma = 0.577\ldots$ is Euler's constant. We will explain in this section why these choices, for any $R_0>0$, ensure a bounded condition number of $A_{k,\eta}$ as $k\to0$ in the case of general Lipschitz $\Gamma$.

To understand this limit we need to recall what is known about integral equation formulations for the Laplace case $k=0$. Let $\Phi_0$ denote the fundamental solution of the Laplace equation,  given simply by \eqref{eq:Phi} with $k=0$ in the 3D case, and defined as in \S\ref{sec:upper} in the 2D case. Let $S_0$ and $D_0$ denote the single and double-layer potentials in the Laplace case, defined by equations \eqref{eq:S} and \eqref{eq:D} with $S$, $D$, and $\Phi$ replaced by $S_0$, $D_0$ and $\Phi_0$. It is a fairly straightforward calculation (see e.g.\ \cite{CWGLL07} for the detail in the case of Lipschitz $\Gamma$) that
\begin{equation} \label{Dk}
\|D_k-D_0\|\to 0 \mbox{ and } \|S_k-S_0\| \to 0
\end{equation}
as $k\to 0$ in the 3D case, and that the first of these results holds also in the 2D case. In the 2D case the limiting behaviour of $S_k$ is more subtle. We see from \eqref{phidif} that
\begin{equation} \label{Sk}
\left\|S_k-S_0 + \frac{1}{2\pi} \log(kR_0) \, T\right\|\to 0
\end{equation}
as $k\to 0$ where $T$ is the finite-rank integral operator defined by
$$
T \phi(x) = \int_\Gamma \phi(y) ds(y), \quad x\in \Gamma.
$$
The following limiting behaviour of $A_{k,\eta}$ is clear from \eqref{Dk} and \eqref{Sk}.

\begin{lemma} \label{lemmalow}
As $k\to0$,
\[
A_{k,\eta} = I + D_0 - \ri\eta S_0(1+o(1))
\]
in 3D, while
\[
A_{k,\eta} = I+D_0 +\ri\eta \frac{1}{2\pi} \log(kR_0) \, T - \ri \eta S_0(1+ o(1))
\]
in 2D. Thus, unless
\[
\eta=\left\{\begin{array}{cc}
              O(1), & d=3, \\
              O((\log k)^{-1}), & d=2,
            \end{array}\right. \quad \mbox{ as } k\to 0,
\]
it holds that $\|A_{k,\eta}\|\to\infty$ as $k\to 0$. On the other hand, if, for some $c_0\in \R$,
\begin{equation} \label{3Dcond}
\eta\to c_0 \quad \mbox{ as } k\to 0,
\end{equation}
in the case $d=3$ or
\begin{equation} \label{2Dcond}
\eta \frac{1}{2\pi} \log(kR_0)\to c_0 \quad \mbox{ as } k\to 0,
\end{equation}
in the case $d=2$, then
\[
\|A_{k,\eta}- A_0\| \to 0 \quad \mbox{ as } k\to 0,
\]
where
\[
A_0:=\left\{\begin{array}{cc}
              I + D_0 - \ri c_0 S_0, & d=3, \\
              I+D_0 +\ri c_0 T, & d=2.
            \end{array}\right.
\]
\end{lemma}

The above lemma, coupled with the following theorem, makes clear that it is appropriate to choose $\eta$ for low $k$ so as to satisfy \eqref{3Dcond} or \eqref{2Dcond}, for $d=3,2$, choosing $c_0\neq 0$ in each case. This choice of $\eta$ ensures that $\cond A_{k,\eta}$ remains bounded in the limit as $k\to 0$. Clearly, one such choice of $\eta$ is \eqref{etachoice}.

\begin{theorem} \label{lowk} Where $A_0$ is as defined in Lemma \ref{lemmalow}, it holds that, for $c_0\neq 0$, $A_0$ is invertible as an operator on $H^s(\Gamma)$ for $0\leq s\leq 1$, in particular as an operator on $L^2(\Gamma)$, while $A_0$ is not invertible for $c_0=0$. Thus, if \eqref{3Dcond} or \eqref{2Dcond} hold in the cases $d=3$ and $d=2$, respectively, then, as $k\to 0$, $\|A_{k,\eta}^{-1}\| =O(1)$, if  $c_0\neq 0$, while $\|A_{k,\eta}^{-1}\| \to\infty$, if  $c_0= 0$.
\end{theorem}
\begin{proof}
The last sentence follows immediately from standard operator perturbation results and \eqref{2Dcond} and \eqref{3Dcond} once the first sentence is proved. In the case $c_0=0$ it is well known that $A_0$ is not injective, having a non-trivial null space which is the set of constant functions, see e.g.\ \cite[Theorem 6.20]{Kress}, \cite{verchota}. To show invertibility of $A_0$ for $c_0\neq 0$ we note first that, by interpolation, it is enough to show invertibility on $H^s(\Gamma)$ for $s=0$ and 1 \cite{mclean}. Further, since the difference $A_0- A_{k,\eta}$ is a compact operator on $L^2(\Gamma)$ and on $H^1(\Gamma)$ (see e.g.\ the proof of Theorem 2.7 in \cite{CWL06}) and since $A_{k,\eta}$ is invertible, it holds that $A_0$ is Fredholm of index zero on $L^2(\Gamma)$ and on $H^1(\Gamma)$, so that it is invertible if and only if it is injective. Moreover, since $A_0$ is
Fredholm with the same index on $H^1(\Gamma)$ and $L^2(\Gamma)$, and $L^2(\Gamma)$ is dense in $H^1(\Gamma)$,
it follows from a standard result on Fredholm operators (see e.g. \cite[\S1]{prsil}), that the null-space of $A_0$ is a subset of $H^1(\Gamma) \supset H^{1/2}(\Gamma)$. In the case that $\Gamma$ is $C^2$ that there are no non-trivial functions in the null-space of $A_0$ in $C(\Gamma)$  is shown in \cite[Theorem 6.24]{Kress} in the case $d=2$ and in \cite[Theorem 3.33]{CK83} in the case $d=3$. In the case when $\Gamma$ is Lipschitz the same arguments can be used to prove injectivity of $A_0$ in $H^{1/2}(\Gamma)$, replacing the mapping properties of layer potentials in classical function spaces in \cite{CK83,Kress} with those in Sobolev spaces in \cite{mclean} (cf.\ the proof of Theorem 2.5 in \cite{CWL06}).
\end{proof}

\subsection{Bounds on condition numbers and choice of $\eta$} \label{sec:cond}
In this section we bring together the results from the sections above and explore their implications for the conditioning of $A_{k,\eta}$, and what this then implies regarding the choice of $\eta$ to minimise $\cond A_{k,\eta}$. We have already noted in \S\ref{sec:starlike} and \S\ref{sec:lowk} recommendations made in the literature regarding the choice of $\eta$, mainly based on study of the case when $\Gamma$ is a circle or sphere. Overwhelmingly (see e.g.\ \cite{Kr:85, KrSp:83, Ami90, Ami93, Gi:97, BrKu:01,BrPC07, DoGrSm:07}) the guidance is to take $\eta$ proportional to $k$ for all but small values of $k$. The choice of $\eta$ for small $k$ has been discussed already in \S\ref{sec:lowk}. One choice of $\eta$, recommended by Kress \cite{Kr:85} for the 3D case, that we have studied in \S\ref{sec:starlike}, is $\eta = \max(1/(2R_0),k)$. This choice, by Lemma \ref{lemmalow} above, is not suitable in the 2D case for low $k$, since with this choice $\|A_{k,\eta}\|\to\infty$ as $k\to0$. An alternative choice, which satisfies \eqref{2Dcond} with $c_0\neq 0$, and which we will use for computations in \S\ref{sec:num},  is
\begin{equation} \label{etachoice2}
\eta := \left\{ \begin{array}{cc}
                  (R_0(1-\log(kR_0))^{-1}, & 0<kR_0 \leq 1, \\
                  k, & kR_0 \geq 1.
                \end{array}\right.
\end{equation}
Here $R_0$ is a length scale of the scatterer $\Omega$; we choose $R_0$ as defined in \S\ref{sec:starlike} in \S\ref{sec:num}. The following theorem, which follows from  \eqref{eqn:normsequal2}, \eqref{Betak}, \eqref{Marko_fb}, \eqref{akb}, and Theorem \ref{lowk}, is a sharpening of results in \cite[Section 6]{CWGLL07}.
\begin{theorem} \label{thmcond1}
Suppose that $\Gamma$ is piecewise $C^2$ and starlike, in the sense of \S\ref{sec:starlike}, and that $\eta$ is given by \eqref{eta_opt} in the case $d=3$, by \eqref{etachoice2} in the case $d=2$. Then, for some constant $C\geq 1$,
$$
1 \leq \|A_{k,\eta}^{-1}\| \leq C, \quad 1 \leq \|A_{k,\eta}\| \leq C(1+k^{(d-1)/2}),
$$
so that
$$
\cond A_{k,\eta} \leq C^2(1+k^{(d-1)/2})
$$
for all $k>0$. In the case $d=2$ we have a sharper lower bound for $k$ large, so that, for some $c>1$,
\begin{equation} \label{sand}
c^{-1}(1+ k^{1/3}) \leq \cond A_{k,\eta} \leq c(1+k^{1/2}),
\end{equation}
for $k>0$.
\end{theorem}

For the case of a circle or sphere we saw in \S\ref{sec:circ} that the above upper bounds are not sharp; with the proposed choices of $\eta$ the sharper bound holds that $$
\cond A_{k,\eta} \leq C(1+k^{1/3}).
$$
We will investigate, in the 2D case, which of the bounds in \eqref{sand} is sharp in \S\ref{sec:num}. We will also investigate the alternative choice for $\eta$ for large $k$ proposed in \cite{BaSa:07}, namely to take $|\eta|= k^{2/3}$. It follows from \eqref{eq:normAsphere} and \eqref{eqn:normsequal2} that, when $\Gamma$ is a sphere, this choice of $\eta$ also implies
$$
\cond A_{k,\eta} \leq C k^{1/3} \quad \mbox{ for } k\geq 1.
$$
We will explore whether this estimate holds for 2D geometries in \S\ref{sec:num}. Note that for starlike polygons and the choices of $\eta$ indicated in Theorem \ref{thmcond1}, it follows from Theorem \ref{th:3} that, for some $c>1$,
\[
c^{-1}(1+ k^{1/2}) \leq \cond A_{k,\eta} \leq c(1+k^{1/2}),
\]
for $k>0$, i.e.\ it is the upper bound in \eqref{sand} that is sharp in this case. We will illustrate this in the numerical results in \S\ref{sec:num}.

For trapping obstacles, in the sense defined in \S\ref{sec:trapping}, faster rates of growth of $\cond A_{k,\eta}$ are inevitable. The following result is deduced in \cite{CWGLL07}, by combining Theorems \ref{th:3} and \ref{th:trapping}.

\begin{theorem} \label{th:condtr1}\cite[equation (6.13)]{CWGLL07} Suppose that the conditions of Theorem \ref{th:trapping} are satisfied. Then, for some $C>0$, where $k_m$ is as defined
in Theorem \ref{th:trapping},
$$
\cond A_{k_m,\eta} \geq C\,k_m^{9/10}\,\left(1+|\eta|(1+k_m)^{-1/2}\right)\left(1+ \frac{|\eta|}{k_m}\right)^{-1}, \quad m\in\N.
$$
Thus, if $\eta = c(1+k_m^p)$, for some constants $c$ and $p$, then, for some constant $\tilde C>0$,
$$
\cond A_{k_m,\eta} \geq \tilde C(1 + k_m^q), \quad m\in\N,
$$
with $q=9/10$ for $0\leq p\leq 1/2$, $q = p + 4/10$, for $1/2\leq p\leq 1$, and $q=14/10$ for $p\geq 1$.
\end{theorem}

For trapping obstacles satisfying the conditions of Theorem \ref{th:trapping2} the situation with regard to conditioning is much worse: the condition number must grow exponentially as $k$ increases through some sequence of wavenumbers.

\begin{theorem} \label{th:trapcond} Suppose that the conditions of Theorem \ref{th:trapping2} are satisfied. Then there exists a sequence $0< k_0< k_1<k_2<\ldots$, with $k_m\to\infty$ as $m\to\infty$, such that, for some $\gamma >0$ and $C>0$,
\[
\cond A_{k_m ,\eta} \geq C \re^{\gamma k_m},
\]
for
$\eta\in\R\setminus\{0\}$ and $m = 0,1,2,\ldots\,$.
\end{theorem}
\begin{proof} By Lemma \ref{lem} and Theorems \ref{th:5} and \ref{th:trapping2}, there exists a sequence $0< k_0< k_1<k_2<\ldots$, with $k_m\to\infty$ as $m\to\infty$, such that, for some $\gamma >0$ and $C>0$,
$$
\cond A_{k_m,\eta} \geq C \min\left(1,|\eta| k_m^{-2/3}\right) \, \re^{\gamma k_m}\, \left(1+ \frac{\vert\eta \vert}{k_m} \right)^{-1}.
$$
But for $|\eta|\leq k_m$ this implies that $\cond A_{k_m,\eta} \geq \frac{1}{2}C \re^{\gamma k_m}$ while for $|\eta| \geq k_m$ this implies that $\cond A_{k_m,\eta} \geq \frac{1}{2}C k_m^{2/3}\re^{\gamma k_m}$, and the result follows.
\end{proof}

\section{Discrete level}
\label{sec:discrete}
\setcounter{equation}{0}

In this section we explore the relationship between $\|A_{k,\eta}\|$ and $\|A_{k,\eta}^{-1}\|$ and the norms of discrete versions of these operators, specifically the norms of matrices arising from Galerkin discretisations.

Let $X_N\subset L^2(\Gamma)$ be a finite-dimensional subspace with $P_N:L^2(\Gamma)\rightarrow X_N$ the corresponding orthogonal projection.  Let $V$ be a bounded linear operator on $L^2(\Gamma)$.  Then, given $y\in L^2(\Gamma)$, a Galerkin method for solving the equation
\[ Vx=y \]
for $x\in L^2(\Gamma)$, is to seek $x_N\in X_N$ such that
\begin{equation}
  P_N V x_N = P_N y.
  \label{eqn:gal}
\end{equation}
Let $\{\phi_1, \ldots, \phi_N\}$ be an orthonormal basis of $X_N$,  define $V_N:X_N\rightarrow X_N$ by $V_N:=P_N V|_{X_N}$, and let $T_N:X_N\rightarrow \C^N$ be defined by
\[ T_N x = [(x,\phi_1) \cdots (x,\phi_N)]^T. \]
Then $T_N$ is an isomorphism, indeed an isometric isomorphism if we give $\C^N$ the standard Euclidean norm $\|\cdot\|_2$.  Further (\ref{eqn:gal}) is equivalent to
\[ V^NT_N x_N = T_N P_N y, \]
where
\[ V^N := T_N V_N T_N^{-1} \]
is the linear operator on $\C^N$ whose matrix representation (that we denote also by $V^N$) is the Galerkin matrix $V^N = [ (V\phi_j,\phi_i) ]$.  Clearly
\begin{equation}
  \|V_N\|=\|V^N\|
  \label{eqn:VN}
\end{equation}
(where we use $\|\cdot\|$ on the right hand side to denote the matrix norm induced by the vector norm $\|\cdot\|_2$), since both $T_N$ and $T_N^{-1}$ are isometries.
Also $V_N$ is invertible if and only if $V^N$ is invertible and, if they are both invertible, then
\[ \|V_N^{-1}\| = \|(V^N)^{-1}\|. \]
Now we need to determine the relationship between $\|V_N\|$ and $\|V\|$.  We first require the following result.
\begin{lemma} \label{lemma:PNWPN}
  If $W$ is a bounded linear operator on $L^2(\Gamma)$ and $P_1, P_2, \ldots$ is a sequence of orthogonal projection operators with $P_N\phi\rightarrow\phi$ for all $\phi\in L^2(\Gamma)$, then
  \[ \|P_N W P_N\| \rightarrow \|W\|, \]
  as $N\rightarrow\infty$.
\end{lemma}
\begin{proof}
  Let $\lambda=\liminf_{N\rightarrow\infty}\|P_N W P_N\|$ and choose a monotonic increasing  sequence $N_1, N_2, \ldots$ of natural numbers with $\|P_{N_k} W P_{N_k}\| \rightarrow\lambda$ as $k\rightarrow\infty$.  Then, for every $\phi\in L^2(\Gamma)$,
  \[ \|W\phi\| = \lim_{k\rightarrow\infty}\|P_{N_k} W P_{N_k}\phi\| \leq \lim_{k\rightarrow\infty}\|P_{N_k} W P_{N_k}\|\|\phi\| = \lambda \|\phi\|, \]
  and hence $\|W\| \leq \lambda = \liminf_{N\rightarrow\infty}\|P_N W P_N\|$.  On the other hand, we have $$\limsup_{N\rightarrow\infty}\|P_N W P_N\| \leq \|W\|,$$ since $\|P_N W P_N\| \leq \|P_N\| \|W\| \|P_N\| = \|W\|$ for every $N$, and hence
  \[ \|W\|\leq \liminf_{N\rightarrow\infty}\|P_N W P_N\| \leq \limsup_{N\rightarrow\infty}\|P_N W P_N\| \leq \|W\|. \]
  Thus $\lim_{N\rightarrow\infty}\|P_N W P_N\|$ exists and is equal to $\|W\|$.
\end{proof}


\vspace{1ex}
\noindent Clearly $\|P_N V P_N\|=\|V_N\|$, so that it follows from~(\ref{eqn:VN}) and Lemma~\ref{lemma:PNWPN} that
\begin{equation} \label{eq:conv_norm}
\|V^N\| = \|V_N\|   = \|P_N V P_N\| \rightarrow \|V\|,
\end{equation}
as $N\rightarrow\infty$.

In the case that $V=I+C$ with $C$ compact, it holds moreover that $\|(V^N)^{-1}\|\to \|V^{-1}\|$ as $N\to\infty$ if $V$ is invertible. To see this, note first that, by Lemma~\ref{lemma:PNWPN}, if $V$ is invertible,
\[ \|P_N V^{-1} P_N \|\rightarrow \|V^{-1}\|, \]
as $N\rightarrow\infty$. Next, let $\tilde V_N = I + P_NC$ and note that
\[ V_N = \tilde V_N|_{X_N}, \]
  and that, since $P_N$ converges strongly (i.e.\ pointwise) to the identity and $C$ is compact, $P_N C$ converges in norm to $C$, so that $\|V-\tilde V_N\|\to 0$.  It follows from standard operator convergence results that $\tilde{V}_N$ is invertible for all sufficiently large $N$.
   But then it follows that also $V_N = \tilde V_N|_{X_N}$ is invertible (as an operator on $X_N$). Indeed, injectivity of $V_N$ is clear by injectivity of $\tilde V_N$. To see surjectivity, take $\psi\in X_N$ and note that, by surjectivity of $\tilde V_N$, there exists a $\phi\in L^2(\Gamma)$ with $\psi=\tilde V_N\phi=\phi+P_NC\phi$, so that $\phi=\psi-P_NC\phi\in X_N$ and hence, $V_N\phi=\psi$.
This argument also shows that $V_N^{-1} = \tilde V_N^{-1}|_{X_N}$.
 Further, $\|\tilde{V}_N^{-1} - V^{-1}\|\rightarrow 0$ as $N\rightarrow\infty$, so that
\begin{eqnarray*}
  \|P_N V^{-1} P_N - P_N \tilde{V}_N^{-1} P_N \| & = & \|P_N(V^{-1} - \tilde{V}_N^{-1}) P_N \| \\
    & \leq & \|P_N\| \|V^{-1} - \tilde{V}_N^{-1}\| \|P_N \| \\
    & = & \|V^{-1} - \tilde{V}_N^{-1}\| \rightarrow 0,
\end{eqnarray*}
as $N\rightarrow\infty$.  Hence, as $N\rightarrow\infty$,
\begin{equation}
  \|(V^N)^{-1}\| = \|V_N^{-1}\| = \|P_N\tilde{V}_N^{-1}P_N\| = \|P_NV^{-1}P_N\| + o(1) \rightarrow \|V^{-1}\|.
  \label{eqn:VN1}
\end{equation}

Equation (\ref{eq:conv_norm}) applies to the Galerkin boundary element method discretisation of all the operators we have discussed in the previous sections, in particular to $S_k$, $D_k$, and $A_{k,\eta}$, provided that the sequence of approximation spaces $X_N$ is chosen so that $P_N$ converges pointwise to the identity. It is enough to check that this pointwise convergence holds on some dense subset, for example to check that
\[
\|P_N\phi-\phi\| = \inf_{\phi_N\in X_N} \|\phi_N-\phi\|\to 0 \mbox{ as } N\to\infty, \;\;\mbox{ for every } \phi\in C(\Gamma).
\]
 Equation (\ref{eqn:VN1}) applies to the operator $A_{k,\eta}$ if $\Gamma$ is $C^1$, for then $A_{k,\eta}$ has the form $A_{k,\eta} = I + C_{k,\eta}$ with $C_{k,\eta}$ compact \cite{Fabes78}.  For general Lipschitz $\Gamma$, it is not known whether~(\ref{eqn:VN1}) holds, indeed it is not even known for any Galerkin method that $V^N$ is invertible for all sufficiently large~$N$.

\section{Numerical results} \label{sec:num}
\setcounter{equation}{0}

In this section we compute $\|V^N\|$ for $V=S_k, D_k,
A_{k,\eta}$, and $\|(V^N)^{-1}\|$ for $V=A_{k,\eta}$, each for a
variety of obstacles, and we compare the computed values with the
upper and lower bounds on the corresponding continuous operators as
described in~\S\ref{sec:previous}.  The aim is to provide
supporting evidence for some of the theoretical results described in
\S\ref{sec:previous}, both quantitative and asymptotic, and to give
some indication of which of the upper and lower bounds may be sharper,
particularly when there is a significant gap between them.  We also
seek an indication of the extent to which the bounds on the continuous
operators are satisfied by their discrete counterparts.

 We present results for $\eta=k$ for all geometries under consideration, and we also present results for
$\eta=k^{2/3}$ for certain specific examples.  As we have discussed in \S\ref{sec:cond}, the choice $\eta=k$ is widespread in the literature, e.g.\ \cite{Ami90, Ami93, DoGrSm:07, Gi:97,Kr:85}, and this choice is supported by our own preceding analysis. The interesting choice $\eta=k^{2/3}$, proposed in \cite{BaSa:07}, is also supported by some of the above analysis; for example we have seen in \S\ref{sec:cond} that, for a spherical scatterer, $\cond A_{k,\eta}$ increases at the same rate as $k\to\infty$ whether $\eta$ is proportional to $k$ or proportional to $k^{2/3}$.

Although our main focus is on larger values of $k$, for two examples we also investigate the limit $k\to0$, presenting results for $\eta=k$ and for $\eta$ given by \eqref{etachoice2}.

In each example the boundary $\Gamma$
is piecewise $C^{\infty}$, that is $\Gamma=\bigcup_{j=1}^p \Gamma^{(j)}$ with $\Gamma^{(j)}$ a $C^\infty$
arc. We denote the length of  $\Gamma^{(j)}$ by
$L_j$, and divide each $\Gamma^{(j)}$ into $N_j$ segments $\Gamma_i^{(j)}$, $i=1,\dots,N_j$ of equal
length $|\Gamma_i^{(j)}|=L_j/N_j$. We then
define the orthonormal basis functions by
\[ \tilde{\phi}_i^{(j)}(x) =\left\{\begin{array}{cc} 1/|\Gamma_i^{(j)}|^{1/2},
    & x\in\Gamma_i^{(j)},
    \\ 0, & \mbox{otherwise},
  \end{array}\right. \quad i=1,\ldots,N_j. \]

Now whilst in theory $\|V\| = \lim_{N\rightarrow\infty} \|V^N\|$, in
practice we can only compute $\|V^N\|$ for a finite value of $N$. In
order to justify the assumption that our choice of $N_j$ is
sufficiently large we fix $N_j\propto k$, choosing the constant of
proportionality on the basis of some simple model experiments.  In
particular, for the case that $\Omega$ is a circle the eigenvalues
of $A_{k,\eta}$ are known explicitly, with corresponding formulae
in terms of the eigenvalues for $\|\Ak\|$ and $\|\Ak^{-1}\|$ (see \cite[\S2]{CWGLL07} for details).
Thus for a circle we can compare our computed approximations to
$\|A_{k,\eta}\|$ and $\|A_{k,\eta}^{-1}\|$ with the known values.
For this example we found that ten basis functions per wavelength gives a
relative error of approximately 1\%, and thus in each example we
choose $N_j\approx\frac{10kL_j}{2\pi}$.

We present numerical results for the obstacles shown in
Figure~\ref{fig:shapes1}.  More detailed descriptions of the
obstacles are provided below.
\begin{figure}[ht]
\center
  \includegraphics[width=13cm]{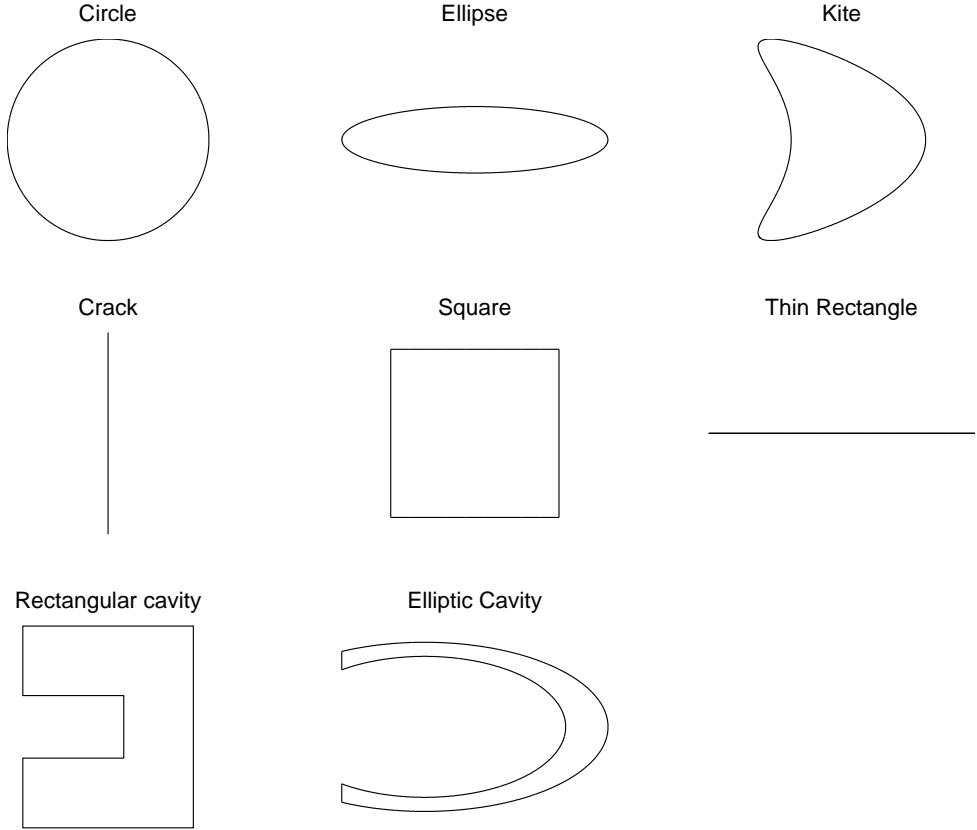}
  \caption{Obstacles corresponding to numerical experiments.}
  \label{fig:shapes1}
\end{figure}
For each obstacle and for each operator $V_k$ we also compute the algebraic growth rate $p$
under the assumption that $\|V_k\|=C k^p$, for some constant $C>0$.  Assuming this formula holds, we can estimate the value of $p$ from two successive values $\|V_{k_j}\|$ and $\|V_{k_{j+1}}\|$ by
\begin{equation}
p=\frac{\log\left(\frac{\|V_{k_{j+1}}\|}{\|V_{k_j}\|}\right)}{\log\left(\frac{k_{j+1}}{k_j}\right)}.
  \label{eqn:EOC}
\end{equation}
In the cases where $p\approx 0$, $p$ is not shown in the tables.  In all of the estimates detailed below, $C$ and $C_j$, $j=1,2,\ldots$, denote unspecified constants independent of $k$ and $\eta$.

\subsection{Circle}
For our first example, we consider the unit circle.  From~(\ref{eq:Sklower}) and~\eqref{circ:SDnorm} we know that, for $k\geq 1$
\[ (32\pi)^{-1/3} k^{-2/3}(1+o(1)) \leq \|S_k\| \leq Ck^{-2/3}, \]
where here and throughout this section $o(1)$ denotes a term which vanishes in the limit as $k\to\infty$.
Theorem~\ref{th:7}, \eqref{circ:SDnorm} and (\ref{eq:Dk2Dupper}) imply that, for $k> 0$,
\[ C_1\leq \|D_k\| \leq C_2k^{1/2} + C_3, \] whilst we know that the sharper  upper bound (\ref{eq:DSboundsSphere}) holds in the case of a sphere, that $\|D_k\| \leq C$.  The numerical results in Table~\ref{table:circle1} for the corresponding boundary element matrices suggest that this sharper result, proved for a sphere, appears to be applicable for a circle as well; we observe for the discrete approximations that $\|S_k\|\sim k^{-2/3}$ and $\|D_k\|\sim k^0$ ($\sim$ in this section indicates that the ratio of the left hand side to the right hand side is approximately constant in the limit $k\to\infty$).  The quantitative lower bound on $\|S_k\|$ from (\ref{eq:Sklower}) is clearly a lower bound in Table~\ref{table:circle1}, underestimating the true norm by a factor of about 6.5.
\begin{table}[htbp]
\begin{center}
  \begin{tabular}{|r|rrr|r|}
  \hline
    $k$ & $(32\pi)^{-1/3} k^{-2/3}$ & $\|S_k\|$ & $p$ & $\|D_k\|$  \\
  \hline
    5 &7.355$\times10^{-2}$ & 5.240$\times10^{-1}$ &       & 1.144  \\
   10 &4.633$\times10^{-2}$ & 3.152$\times10^{-1}$ & -0.73 & 1.114  \\
   20 &2.919$\times10^{-2}$ & 1.997$\times10^{-1}$ & -0.66 & 1.084  \\
   40 &1.839$\times10^{-2}$ & 1.246$\times10^{-1}$ & -0.68 & 1.079  \\
   80 &1.158$\times10^{-2}$ & 7.798$\times10^{-2}$ & -0.68 & 1.076  \\
  160 &7.297$\times10^{-3}$ & 4.884$\times10^{-2}$ & -0.68 & 1.075  \\
  320 &4.597$\times10^{-3}$ & 3.076$\times10^{-2}$ & -0.67 & 1.072  \\
  640 &2.896$\times10^{-3}$ & 1.935$\times10^{-2}$ & -0.67 & 1.071   \\
  \hline
  \end{tabular}
  \caption{{\bf Circle.} Norms of Galerkin BEM approximations to $S_k$ and $D_k$, and $p$ values given by \eqref{eqn:EOC}.}
  \label{table:circle1}
  \end{center}
\end{table}

From Lemma~\ref{lem}, Corollary~\ref{cor45}, (\ref{eq:DoGrSmnorm}) and (\ref{Marko_fb}) we know that, for $k\geq 1$,
\begin{eqnarray*}
   \left(\frac{k}{32\pi}\right)^{1/3}(1+o(1)) &\leq \|A_{k,k}\| \leq& C_1k^{1/3}, \\
   1 &\leq \|A_{k,k^{2/3}}\| \leq& C_2k^{1/2}.
\end{eqnarray*}
Note that Corollary~\ref{cor45} gives the lower bound $(32\pi)^{-1/3} - 1/(2\sqrt{2}) \approx -0.139$ on $\|A_{k,k^{2/3}}\|$, as $k\to\infty$, which is clearly less sharp than the lower bound on $\|A_{k,k^{2/3}}\|$ from Lemma~\ref{lem}.  The upper bound on $\|A_{k,k^{2/3}}\|$ for the case of a sphere is, from~(\ref{eq:normAsphere}), $\|A_{k,k^{2/3}}\| \leq C_3$.  The numerical results in Table~\ref{table:circle2} suggest that this sharper result also holds for a circle; the results suggest $\|A_{k,k^{2/3}}\|\sim k^0$, and that  $\|A_{k,k}\|\sim k^{1/3}$ as expected. The quantitative lower bound on $\|A_{k,k}\|$ from Corollary~\ref{cor45} is a lower bound in Table~\ref{table:circle2}, underestimating the true norm by a factor of about 7.
\begin{table}[htbp]
\begin{center}
  \begin{tabular}{|r|crr|c|c|ccr|}
  \hline
    $k$ & $\left(\frac{k}{32\pi}\right)^{1/3}$ & $\small\|A_{k,k}\|$ & $\small p$ & $\small\|A_{k,k}^{-1}\|$ &  $\small \|A_{k,k^{2/3}}\|$ &
    $\small\|A_{k,k^{2/3}}^{-1}\|$ & $\small p$ & $\small B_{0,k^{2/3}}$ \\
  \hline
    5 & 0.37 & 2.663 &      & 0.986 & 2.016 &  0.995 &     & 3.82 \\
   10 & 0.46 & 3.233 & 0.28 & 0.987 & 1.993 &  1.056 &0.09 & 4.49 \\
   20 & 0.58 & 4.021 & 0.32 & 0.987 & 1.981 &  1.260 &0.26 & 5.38 \\
   40 & 0.74 & 5.030 & 0.32 & 0.987 & 2.000 &  1.701 &0.43 & 6.56 \\
   80 & 0.93 & 6.271 & 0.32 & 0.987 & 1.999 &  2.039 &0.26 & 8.06 \\
  160 & 1.17 & 7.859 & 0.33 & 0.987 & 1.990 &  2.694 &0.40 & 9.98 \\
  320 & 1.47 & 9.883 & 0.33 & 0.987 & 1.998 &  3.407 &0.34 & 12.40\\
  640 & 1.85 & 12.419 & 0.33 & 0.987 & 2.000 &  4.307&0.34 & 15.49 \\
  \hline
  \end{tabular}
  \caption{{\bf Circle.} Galerkin BEM approximations to $\|A_{k,\eta}\|$ and  $\|A_{k,\eta}^{-1}\|$.}
  \label{table:circle2}
\end{center}
\end{table}

By Lemma~\ref{lem}, $\|A_{k,k}^{-1}\|\geq 1$, which combined with \eqref{eq:DoGrSmnorminv2} implies that $\|A_{k,k}^{-1}\|=1$ for all  $k$ sufficiently large, and the numerical results in Table~\ref{table:circle2} show this behaviour. The bound for general starlike obstacles applied to the circle, i.e.\ (\ref{eqn:Bcircle}), gives that
\[ \|A_{k,\eta}^{-1}\| \leq \frac12 + \left[1 + \frac{k^2}{\eta^2} + \frac{(1+2k)^2}{2\eta^2}\right]^{1/2} =: B_{0,\eta}. \]
Note that $B_{0,k}\rightarrow2.5$ (in fact it holds that $2.5\leq B_{0,k}\leq 2.6$ for the range of $k$ in Table \ref{table:circle2}), and that $B_{0,k^{2/3}}\sim \sqrt{3}\, k^{1/3}$ as $k\to\infty$. We see from Table~\ref{table:circle2} that $B_{0,\eta}$ appears to be  an upper bound for the discretisation of $\|A_{k,\eta}^{-1}\|$ as predicted, overestimatimating by a factor of about 2.5 for the larger values of $k$ when $\eta=k$, by a factor of about 3.6 when $\eta=k^{2/3}$.

We note from Table \ref{table:circle2} that, for this example, the condition number $\cond A_{k,\eta} = \|A_{k,\eta}\|\, \|A_{k,\eta}^{-1}\|$ appears to be slightly numerically smaller for $\eta=k^{2/3}$ than for $\eta=k$. It appears that, for both choices of $\eta$, $\cond A_{k,\eta}$ increases approximately in proportion to $k^{1/3}$, though this is less clear in the case $\eta = k^{2/3}$.

\subsection{Ellipse}
Next we consider the ellipse given by $(2\cos t,\frac12\sin t)$, $t\in[0,2\pi]$.
The more specific results of \S\ref{sec:circ} do not apply in this case, and for upper bounds on $\|S_k\|$ and $\|D_k\|$ we have only the results for general Lipschitz $\Gamma$ of \S\ref{sec:upper}.
The inequalities (\ref{eq:Sklower}) and (\ref{eq:Sk2Dupper}) imply that, for $k\geq 1$,
\[ (4\pi)^{-1/3}k^{-2/3}(1+o(1)) \leq \|S_k\| \leq Ck^{-1/2},  \]
the lower bound larger than for the case of the circle as the maximum radius of curvature ($R=8$) is larger.
Theorem~\ref{th:7} and (\ref{eq:Dk2Dupper}) with $N=0$ imply that, for $k> 0$,
\[ C_1\leq \|D_k\| \leq C_2k^{1/2} + C_3. \]
Inspecting the numerical results in Table \ref{table:ellipse1}, we see that the quantitative lower bound on $\|S_k\|$ from (\ref{eq:Sklower}) is clearly a lower bound for the norm of the discretised operator, underestimating the true norm by a factor of about 6 at the highest wavenumbers (cf.\ the results for the circle).
 The numerical results for $\|D_k\|$ suggest that $\|D_k\|\sim k^0$, i.e.\ that the lower bound on $\|D_k\|$ is sharp, while it appears from the numerical results that $\|S_k\|\sim k^p$, for $p\approx -0.6$.
\begin{table}[htbp]
\begin{center}
  \begin{tabular}{|r|rrr|rr|}
  \hline
    $k$ & $(4\pi)^{-1/3}k^{-2/3}$ & $\|S_k\|$ & $p$ & $\|D_k\|$ & p\\
  \hline
    5 & 1.471$\times10^{-1}$ & 6.692$\times10^{-1}$ &       & 1.458 &  \\
   10 & 9.267$\times10^{-2}$ & 4.143$\times10^{-1}$ & -0.69 & 1.591& 0.13  \\
   20 & 5.838$\times10^{-2}$ & 2.730$\times10^{-1}$ & -0.60 & 1.671  & 0.07 \\
   40 & 3.678$\times10^{-2}$ & 1.803$\times10^{-1}$ & -0.60 & 1.760 & 0.08 \\
   80 & 2.317$\times10^{-2}$ & 1.209$\times10^{-1}$ & -0.58 & 1.819 & 0.05  \\
  160 & 1.459$\times10^{-2}$ & 8.029$\times10^{-2}$ & -0.59 & 1.877 & 0.05 \\
  320 & 9.194$\times10^{-3}$ & 5.269$\times10^{-2}$ & -0.61 & 1.919  & 0.03 \\
  640 & 5.792$\times10^{-3}$ & 3.427$\times10^{-2}$ & -0.62 & 1.942 & 0.02\\
  \hline
  \end{tabular}
  \caption{{\bf Ellipse.} Galerkin BEM approximations to  $\|S_k\|$ and $\|D_k\|$.}
  \label{table:ellipse1}
  \end{center}
\end{table}

Now turning to Table \ref{table:ellipse2}, note that   Lemma~\ref{lem}, Corollary~\ref{cor45} and (\ref{Marko_fb}) imply that, for $k\geq 1$,
\begin{eqnarray}
   \left(\frac{k}{4\pi}\right)^{1/3}(1+o(1)) &\leq \|A_{k,k}\| \leq& C_1k^{1/2}, \nonumber \\ 
   1 &\leq \|A_{k,k^{2/3}}\| \leq& C_2k^{1/2}\label{eq:ellipseeq2}.
\end{eqnarray}
(Note that Corollary~\ref{cor45} gives the lower bound $(4\pi)^{-1/3} - 1/(2\sqrt{2}) \approx 0.077$ on $\|A_{k,k^{2/3}}\|$, as $k\to\infty$. So in \eqref{eq:ellipseeq2} we have used the sharper estimate from Lemma~\ref{lem}.)  The numerical results in Table~\ref{table:ellipse2} suggest that  $\|A_{k,k}\|\sim k^{p}$ for $p\approx 0.4$, that $\|A_{k,k^{2/3}}\|\sim k^{0}$, and  that the lower bound in \eqref{eq:ellipseeq2} is an underestimate by a factor approximately 2.5.
\begin{table}[htbp]
\begin{center}
  \begin{tabular}{|r|rrr|r|r|rr|}
  \hline
    $k$ & $\left(\frac{k}{4\pi}\right)^{1/3}$ & $\|A_{k,k}\|$ & $p$ & $\|A_{k,k}^{-1}\|$ & $\|A_{k,k^{2/3}}\|$ & $\|A_{k,k^{2/3}}^{-1}\|$ & $p$ \\
  \hline
   5 & 0.736 &  3.507  &      & 0.987 &  2.417 &  0.996 &       \\
  10 & 0.927 &  4.267  & 0.28 & 0.987 &  2.473 &  1.024 &  0.04    \\
  20 & 1.168 &  5.589  & 0.39 & 0.987 &  2.554 &  1.300 &  0.34 \\
  40 & 1.471 &  7.317  & 0.39 & 0.987 &  2.599 &  1.662 &  0.35 \\
  80 & 1.853 &  9.751  &  0.41 & 0.987 & 2.580 &  1.986 &  0.26 \\
 160 & 2.335 & 12.902  & 0.40 & 0.987 &  2.582 &  2.548 &  0.36 \\
 320 & 2.942 & 16.906  & 0.39 & 0.987 &  2.634 &  3.387 &  0.41 \\
 640 & 3.707 & 21.972  & 0.38 & 0.987 &  2.689 &  4.287 &  0.34 \\
  \hline
  \end{tabular}
  \caption{{\bf Ellipse.} Galerkin BEM approximations to $\|A_{k,\eta}\|$ and  $\|A_{k,\eta}^{-1}\|$.}
  \label{table:ellipse2}
\end{center}
\end{table}

Lemma~\ref{lem} and~(\ref{eqn:normsequal2}) imply that, for $k\geq 1$,
\begin{eqnarray} \label{si}
  1 \leq \|A_{k,k}^{-1}\| \leq C_1, \quad
  1 \leq \|A_{k,k^{2/3}}^{-1}\| \leq C_2k^{1/3}.
\end{eqnarray}
The numerical results in Table~\ref{table:ellipse2} suggest that $\|A_{k,k}^{-1}\|\approx 1$ for $k\geq 5$.
The values of $p$ corresponding to  $\|A_{k,k^{2/3}}^{-1}\|$ are rather variable, but the average of the last six values for $p$ is $0.34$, approximately consistent with the upper bound $C_2 k^{1/3}$.

As for the case of the circle the condition number $\cond A_{k,\eta} = \|A_{k,\eta}\|\, \|A_{k,\eta}^{-1}\|$ appears to be numerically smaller for $\eta=k^{2/3}$ than for $\eta=k$, by nearly a factor 2 for the higher values of $k$ in Table \ref{table:ellipse2}.


\subsection{Kite} \label{sec:kite}
\begin{figure}[ht]
\center
  \includegraphics[width=8cm]{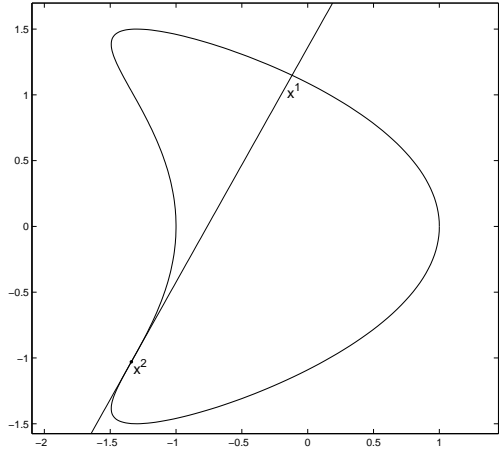}
  \caption{The `kite' shape and notation of \S\ref{sec:kite}. }
  \label{fig:kite}
\end{figure}
Estimates for $\|S_k\|$, $\|D_k\|$, $\|A_{k,k}\|$,
$\|A_{k,k}^{-1}\|$, $\|A_{k,k^{2/3}}\|$ and $\|A_{k,k^{2/3}}^{-1}\|$
for the `kite' shape given by $(\cos t+0.65\cos 2t -0.65, 1.5\sin
t)$, $t\in[0,2\pi]$, are shown in Tables~\ref{table:kite1}
and~\ref{table:kite2}, for various $k$.  Our upper bounds from \S\ref{sec:previous} that apply in this case are identical in their dependence on $k$ to those for the ellipse. However, in contrast to the circle and the ellipse, our lower bounds imply significant growth in $\|D_k\|$ and $\|A_{k,k^{2/3}}\|$ as $k$ increases.   Specifically, applying Theorem~\ref{th:5} to the non-convex kite, with $N=2$ and $x^*$ the point of inflection on $\Gamma$ labelled $x^2$ in Figure \ref{fig:kite}, and recalling~(\ref{eq:Sk2Dupper}) and (\ref{Marko_fb}), it follows that, for $k\geq 1$,
\begin{eqnarray*}
  C_1 k^{-3/5} & \leq \|S_k\| \leq & C_2 k^{-1/2}, \\
  C_1 k^{2/5}  & \leq \|A_{k,k}\| \leq & C_2 k^{1/2}, \\
  C_1 k^{1/15} & \leq \|A_{k,k^{2/3}}\| \leq & C_2 k^{1/2}.
\end{eqnarray*}
Moreover, applying Theorem~\ref{th:7}  to the non-convex kite, with $N=2$, $x^2$ one of the points of inflection on $\Gamma$, and $x^1$ the other point on $\Gamma$ intersected by the tangent at $x^2$ (see Figure \ref{fig:kite}),  and recalling~(\ref{eq:Dk2Dupper}), we have, for $k\geq 1$,
\[ C_1 k^{1/6} \leq \|D_k\| \leq C_2 k^{1/2}. \]
The numerical results in Tables~\ref{table:kite1} and~\ref{table:kite2} provide some support for these estimates.  The lower bound on $\|S_k\|$ seems sharper than the upper bound, although in fact the behaviour of $\|S_k\|$ appears to be rather similar to that for the ellipse.  On the other hand, the behaviour of  $\|D_k\|$ and $\|A_{k,k^{2/3}}\|$ is very different from that seen for the ellipse, with, approximately, $\|D_k\|\sim k^{1/4}$ and $\|A_{k,k^{2/3}}\|\sim k^{1/6}$.

Since the kite shape is starlike, satisfying the assumptions of \S\ref{sec:starlike}, the same bounds \eqref{si} hold on $\|A_{k,k}^{-1}\|$ and $\|A_{k,k^{2/3}}^{-1}\|$ as for the case of the ellipse. We observe similar behaviour to that of the ellipse, namely that $\|A_{k,k}^{-1}\|\approx 1$ and that the bound $\|A_{k,k^{2/3}}^{-1}\|\leq C_2 k^{1/3}$ appears sharp. In Table \ref{table:kite2} $\cond A_{k,k^{2/3}}\sim k^p$, with $p\approx 0.47$, a faster rate of growth than $\cond A_{k,k}\sim k^{1/3}$, so that, for the larger values of $k$, $\cond A_{k,k} < \cond A_{k,k^{2/3}}$.
\begin{table}[htbp]
\begin{center}
  \begin{tabular}{|r|rr|rr|}
  \hline
    $k$ & $\|S_k\|$ & $p$ & $\|D_k\|$ & $p$ \\
  \hline
    5 & 6.591$\times10^{-1}$ &       &  1.810 &        \\
   10 & 4.365$\times10^{-1}$ & -0.59 &  2.169 &   0.27 \\
   20 & 2.758$\times10^{-1}$ & -0.63 &  2.686 &   0.29 \\
   40 & 1.712$\times10^{-1}$ & -0.65 &  3.160 &   0.27 \\
   80 & 1.074$\times10^{-1}$ & -0.65 &  3.616 &   0.25 \\
  160 & 6.759$\times10^{-2}$ & -0.66 &  4.160 &   0.24 \\
  320 & 4.400$\times10^{-2}$ & -0.65 &  4.760 &   0.24 \\
  640 & 2.866$\times10^{-2}$ & -0.65 &  5.437 &   0.23 \\
  \hline
  \end{tabular}
  \caption{{\bf Kite.} Galerkin BEM approximations to  $\|S_k\|$ and $\|D_k\|$.}
  \label{table:kite1}
  \end{center}
\end{table}

\begin{table}[htbp]
\begin{center}
  \begin{tabular}{|r|rr|r|rr|rr|}
  \hline
    $k$ & $\|A_{k,k}\|$ & $p$ & $\|A_{k,k}^{-1}\|$ & $\|A_{k,k^{2/3}}\|$ & $p$ & $\|A_{k,k^{2/3}}^{-1}\|$ & $p$ \\
  \hline
    5 &  3.720 &      & 0.987 &    2.764  &       &   1.029 &       \\
   10 &  4.766 & 0.36 & 0.987 &    3.243  &  0.23 &   1.027 & 0 \\
   20 &  6.151 & 0.37 & 0.987 &    3.602  &  0.15 &   1.177 &  0.20 \\
   40 &  7.513 & 0.29 & 0.987 &    4.041  &  0.17 &   1.511 &  0.36 \\
   80 &  9.316 &  0.31 & 0.987 &    4.394  &  0.12 &   1.883 &  0.32 \\
  160 & 11.563 & 0.31 & 0.987 &    4.891  &  0.16 &   2.354 &  0.32 \\
  320 & 14.337 & 0.31 & 0.987 &    5.414  &  0.15 &   2.954 &  0.33 \\
  640 & 18.387 & 0.36 & 0.987 &    6.030  &  0.16 &   3.671 &  0.31 \\
  \hline
  \end{tabular}
  \caption{{\bf Kite.} Galerkin BEM approximations to $\|A_{k,k}\|$, $\|A_{k,k}^{-1}\|$, $\|A_{k,k^{2/3}}\|$ and $\|A_{k,k^{2/3}}^{-1}\|$.}
  \label{table:kite2}
\end{center}
\end{table}

\subsection{Crack} This numerical  example is distinct from the others in that $\Gamma$ is an open arc, the straight line from $(0,0)$ to $(0,1)$, and the only theory which applies from \S\ref{sec:previous} are the upper and lower bounds for $\|S_k\|$.
Computations of the Galerkin BEM approximations to $\|S_k\|$ are shown in Table~\ref{table:crack1} for various
$k$.  From~(\ref{eqn:Skcrack}) and Theorem~\ref{th:3} it follows that
\begin{equation}
\label{eq:crack1}
\frac{1}{\sqrt{\pi k}} + O(k^{-1}) \leq \|S_k\| \leq \frac{2}{\sqrt{\pi k}}.
\end{equation}
The results in Table~\ref{table:crack1} clearly demonstrate that $\|S_k\|\sim k^{-1/2}$ (cf.\ \cite{WaChew04}), a slower rate of decay
than for the circle, ellipse or kite, and that the values of $\|S_k\|$ are bracketed between the quantitative upper and lower bounds in \eqref{eq:crack1}, nearly coinciding with the lower bound values.
\begin{table}[htbp]
\begin{center}
  \begin{tabular}{|r|rr|rr|}
  \hline
    $k$ & $\|S_k\|$ & $p$ & $\sqrt{1/(\pi k)}$ & $2 \sqrt{1/(\pi k)}$\\
  \hline
    5 & 2.649$\times10^{-1}$ &       & 2.523$\times10^{-1}$ & 5.046$\times10^{-1}$ \\
   10 & 1.817$\times10^{-1}$ & -0.54 & 1.784$\times10^{-1}$ & 3.568$\times10^{-1}$  \\
   20 & 1.266$\times10^{-1}$ & -0.52 & 1.262$\times10^{-1}$ & 2.523$\times10^{-1}$  \\
   40 & 8.960$\times10^{-2}$ & -0.50 & 8.921$\times10^{-2}$ & 1.784$\times10^{-1}$  \\
   80 & 6.326$\times10^{-2}$ & -0.52 & 6.308$\times10^{-2}$ & 1.262$\times10^{-1}$  \\
  160 & 4.472$\times10^{-2}$ & -0.50 & 4.460$\times10^{-2}$ & 8.921$\times10^{-2}$  \\
  320 & 3.162$\times10^{-2}$ & -0.50 & 3.154$\times10^{-2}$ & 6.308$\times10^{-2}$  \\
  640 & 2.236$\times10^{-2}$ & -0.50 & 2.230$\times10^{-2}$ & 4.460$\times10^{-2}$  \\
  \hline
  \end{tabular}
  \caption{{\bf Crack.} Galerkin BEM approximations for $\|S_k\|$, and theoretical lower and upper bounds.}
  \label{table:crack1}
  \end{center}
\end{table}

\subsection{Square}
Computed estimates for $\|S_k\|$, $\|D_k\|$, $\|A_{k,k}\|$,
$\|A_{k,k}^{-1}\|$, $\|A_{k,k^{2/3}}\|$ and $\|A_{k,k^{2/3}}^{-1}\|$
for the square of side length two are shown in
Tables~\ref{table:square1} and~\ref{table:square2} below, for
various $k$.  The theoretical upper bounds from \S\ref{sec:previous} that apply are identical to those for the ellipse and kite examples above. In particular, since the square is starlike satisfying the conditions of \S\ref{sec:starlike}, the bounds \eqref{si} apply.  However, as $\Gamma$ now contains a straight line segment we have different lower bounds on $\|S_k\|$, $\|A_{k,\eta}\|$ and $\|D_k\|$ compared to the ellipse and the kite.  Specifically, applying Theorem~\ref{th:3} and recalling~(\ref{eq:Sk2Dupper}) and (\ref{Marko_fb}), it follows that, for $k\geq 1$,
\begin{eqnarray}
  \sqrt{\frac{2}{\pi k}} + O(k^{-1}) & \leq \|S_k\| \leq & C k^{-1/2}\label{eq:square1}, \\
  \sqrt{\frac{2k}{\pi}} -1 + O(1) & \leq \|A_{k,k}\| \leq & C k^{1/2}\label{eq:square2}, \\
  \sqrt{\frac{2}{\pi}}k^{1/6} -1 + O(k^{-1/3}) & \leq \|A_{k,k^{2/3}}\| \leq & C k^{1/2}\label{eq:square3}.
\end{eqnarray}
 Applying Theorem~\ref{th:6} and recalling~(\ref{eq:Dk2Dupper}), we also have
\begin{equation}
\label{eq:square4}
C_1 k^{1/4} \leq \|D_k\| \leq C_2 k^{1/2}.
\end{equation}
It appears from Table~\ref{table:square1} that $\|S_k\|\sim k^{-1/2}$, as expected, and the quantitative lower bound in \eqref{eq:square1} appears to be sharp, underestimating $\|S_k\|$ by only $3\%$ at the highest frequency.
It also seems that $\|D_k\|\sim k^{1/4}$, indicating that the lower bound in \eqref{eq:square4} is sharp in its dependence on $k$.
\begin{table}[htbp]
\begin{center}
  \begin{tabular}{|r|rrr|rr|}
  \hline
    $k$ & $\sqrt{\frac{2}{\pi k}}$ & $\|S_k\|$ & $p$ & $\|D_k\|$ & $p$ \\
  \hline
    5 & 3.568$\times10^{-1}$ & 5.784$\times10^{-1}$ &       & 1.316 &      \\
   10 & 2.523$\times10^{-1}$ & 3.353$\times10^{-1}$ & -0.79 & 1.488 & 0.18 \\
   20 & 1.784$\times10^{-1}$ & 2.137$\times10^{-1}$ & -0.65 & 1.730 & 0.22 \\
   40 & 1.262$\times10^{-1}$ & 1.428$\times10^{-1}$ & -0.58 & 2.018 & 0.22 \\
   80 & 8.921$\times10^{-2}$ & 9.760$\times10^{-2}$ & -0.55 & 2.389 & 0.24 \\
  160 & 6.308$\times10^{-2}$ & 6.723$\times10^{-2}$ & -0.54 & 2.825 & 0.24 \\
  320 & 4.460$\times10^{-2}$ & 4.667$\times10^{-2}$ & -0.53 & 3.346 & 0.25 \\
  640 & 3.154$\times10^{-2}$ & 3.259$\times10^{-2}$ & -0.52 & 3.972 & 0.25 \\
  \hline
  \end{tabular}
  \caption{{\bf Square.} Galerkin BEM approximations for  $\|S_k\|$ and $\|D_k\|$.}
  \label{table:square1}
  \end{center}
\end{table}

The results in Table~\ref{table:square2} suggest that $\|A_{k,k}\|\sim k^{p}$, with $p\approx 1/2$, as expected from \eqref{eq:square2}. It appears that
$\|A_{k,k^{2/3}}\|$ is increasing roughly like $k^{1/5}$.  The quantitative lower bound in \eqref{eq:square3} is seen to be a lower bound for the Galerkin BEM discretisation of $\|A_{k,k^{2/3}}\|$ in Table \ref{table:square2}, underestimating $\|A_{k,k^{2/3}}\|$ by about a factor 3.5 at the highest frequency.  As in the cases of the circle, ellipse, and kite, $\|A^{-1}_{k,k}\|\approx 1$ for all $k$, while $\|A_{k,k^{2/3}}^{-1}\|$ is
increasing as  $k$ increases, though the rate of increase is somewhat erratic (from~(\ref{eqn:normsequal2}), we recall that $\|A_{k,k^{2/3}}^{-1}\|\leq Ck^{1/3}$). However, it appears that $\cond A_{k,k^{2/3}}\sim k^p$, with $p\approx 0.6$, a faster rate of growth than $\cond A_{k,k}\sim k^{1/2}$, and, for the largest value of $k$, $\cond A_{k,k} \approx \cond A_{k,k^{2/3}}$.

Where $R_0>0$ is some length scale of the scatterer, it follows from Theorem \ref{thmcond1} that choosing
$$
\eta = \eta^*:=1/(R_0(1-\ln(kR_0)),
$$
ensures $\|A_{k,\eta}\|$ and $\|A_{k,\eta}^{-1}\|$ remain bounded as $k\to0$. This is true for any Lipschitz $\Gamma$ and (rather arbitrarily) we choose this example to illustrate this numerically. Define $R_0$ as in \S\ref{sec:starlike}, so $R_0=\sqrt{2}$ for this particular scatterer (taking the origin at the centre of the square).
With this choice of $R_0$, we show in Table \ref{table:square3} norm computations for small values of $k$. We see that, while $\|A_{k,k}^{-1}\|$ seems to blow up for small values of $k$, the values of $\|A_{k,\eta^*}^{-1}\|$ remain essentially constant, and $\cond A_{k,\eta^*}$ appears to be approaching a limit of about 3.7 as $k\to0$. For the computations in Table \ref{table:square3}, deviating from the element sizes used in the other calculations, we discretised each boundary line of the square with $500$ equal length elements so as to be able to resolve the blow up in the norm of the inverse of $A_{k,k}$ with some accuracy.

\begin{table}[htbp]
\begin{center}
  \begin{tabular}{|r|rr|r|rrr|rr|}
  \hline
    $k$ & $\|A_{k,k}\|$ & $p$ & $\|A_{k,k}^{-1}\|$ & $\sqrt{\frac{2}{\pi}}k^{1/6} -1$ & $\|A_{k,k^{2/3}}\|$ & $p$ & $\|A_{k,k^{2/3}}^{-1}\|$ & $p$ \\
  \hline
    5 &  3.089 &        &   1.024 & 0.043  & 2.237 &        &   1.116 &        \\
   10 &  3.611 &   0.23 &   1.024 & 0.171  & 2.375 &   0.09 &   1.172 &   0.07 \\
   20 &  4.608 &   0.35 &   1.023 & 0.315  & 2.536 &   0.10 &   1.280 &   0.13 \\
   40 &  6.032 &   0.39 &   1.023 & 0.476  & 2.865 &   0.18 &   1.484 &   0.21 \\
   80 &  8.117 &   0.43 &   1.023 & 0.656  & 3.187 &   0.15 &   1.904 &   0.36 \\
  160 & 11.068 &   0.45 &   1.023 & 0.859  & 3.600 &   0.18 &   2.315 &   0.28 \\
  320 & 15.253 &   0.46 &   1.023 & 1.087  & 4.131 &   0.20 &   3.065 &   0.41 \\
  640 & 21.177 &   0.47 &   1.023 & 1.342  & 4.777 &   0.21 &   4.064 &   0.41 \\
  \hline
  \end{tabular}
  \caption{{\bf Square.} Galerkin BEM approximations for  $\|A_{k,k}\|$, $\|A_{k,k}^{-1}\|$, $\|A_{k,k^{2/3}}\|$ and $\|A_{k,k^{2/3}}^{-1}\|$.}
  \label{table:square2}
\end{center}
\end{table}

\begin{table}[htbp]
\begin{center}
  \begin{tabular}{|r|r|rr|r|r|}
  \hline
    $k$ & $\|A_{k,k}\|$ & $\|A_{k,k}^{-1}\|$ & $p$  & $\|A_{k,\eta^*}\|$ & $\|A_{k,\eta^*}^{-1}\|$ \\
  \hline
    $10^{-5}$ & 1.608 & 3684.85 &   & 1.738 & 2.106 \\
    $10^{-4}$& 1.608 & 465.66 &   -0.89 & 1.723 & 2.105 \\
    $10^{-3}$& 1.608 & 61.63 &       -0.88  &  1.703 & 2.104 \\
    $10^{-2}$& 1.608 & 9.04 &       -0.83 &   1.680 & 2.100 \\
    $10^{-1}$& 1.612 & 2.11 &       -0.63 &   1.693 & 2.081 \\
    $1$ &           2.391 & 1.88 &    -0.05   &    2.534 & 1.871 \\
   \hline
  \end{tabular}
  \caption{{\bf Square.} Galerkin BEM approximations for $\|A_{k,k}\|$, $\|A_{k,k}^{-1}\|$, $\|A_{k,\eta^*}\|$ , and $\|A_{k,\eta^*}^{-1}\|$.}

  \label{table:square3}
\end{center}
\end{table}

\subsection{Thin rectangle} The bound \eqref{eqn:normsequal2} suggests that, even if the obstacle $\Omega$ is starlike, $\|A_{k,\eta}^{-1}\|$ may blow up as the obstacle becomes very thin when $\delta_+/\delta_-$ is small. In this example we investigate the extent to which such a blow up happens. Where $\delta$ is the thickness of the scatterer, it is reasonable to expect this to be a problem when $k \delta\ll 1$, for if $\kappa(x,y)$, $x,y\in\Gamma$, denotes the kernel of the integral operator $D_k-\ri\eta S_k$ and $x_\pm\in\Gamma$ are adjacent points on opposite sides of a thin part of $\Gamma$, then $\kappa(x_+,y)\approx \kappa(x_-,y)$, $y\in \Gamma$, so that the integral operator should be badly conditioned. To explore the extent to which this is a problem, and the extent to which the bound \eqref{eqn:normsequal2} reflects actual behaviour in this limit, we show
estimates for
$\|A_{k,k}\|$ and $\|A_{k,k}^{-1}\|$ for a rectangle with side lengths 2 and
0.02 in Table~\ref{table:recteps002_3} below, for various small values of $k$.

Using the notation of~\S\ref{sec:starlike}, for this example we have $\delta^*=\delta^+=2$, $\delta_-=0.02$, and $R_0=\sqrt{4.0004}$, and (\ref{eqn:normsequal2}) tells us that
\[ \|A_{k,\eta}^{-1}\| \leq B, \]
where $B$ is as defined in \S\ref{sec:starlike}. However, we see from Table \ref{table:recteps002_3} that this bound is a gross overestimate, at least provided we choose $\eta$ carefully. For the definition of $B$ implies that, whatever the choice of $\eta\in \R\setminus\{0\}$,
$$
B > \frac{1}{2}+\left[\left(\frac{\delta_+}{\delta_-}+\frac{4{\delta^*}^2}{\delta_-^2}\right)\left(\frac{\delta_+}{\delta_-}
  +\frac{{\delta^*}^2}{\delta_-^2}\right)
  \right]^{1/2} > 2\frac{{\delta^*}^2}{\delta_-^2} = 2\times 10^4
$$
for this geometry. If we choose $\eta=k$ for small $k$ then, indeed, we see significant blowup as $k\to0$, as for the case of the square (indeed the values of $\|A^{-1}_{k,\eta}\|$ are similar). But, if we choose $\eta =\eta^* := 1/{R_0(1-\ln(k R_0))}$, we know from Theorem \ref{thmcond1} that $\|A_{k,\eta}^{-1}\|$ must stay bounded as $k\rightarrow 0$. In fact we see some mild, logarithmic growth in Table \ref{table:recteps002_3}, but $\|A^{-1}_{k,\eta^*}\|$ is never larger than 26 for the range of $k$ shown.

For the computations in Table \ref{table:recteps002_3} we used $100$ elements on each boundary segment. Comparing the algebraic rates $p$ associated with $\|A_{k,k}^{-1}\|$ for
the first three wavenumbers $k=10^{-5},10^{-4},10^{-3}$ it appears possible that the discretisation does not fully resolve $\|A_{k,k}^{-1}\|$ for $k=10^{-5}$ and that the exact value may be higher. However, due to convergence issues of the underlying singular value decomposition for the norm computation no finer discretisation could be used here.

\begin{table}[htbp]
\begin{center}
  \begin{tabular}{|r|r|rr|r|rr|}
  \hline
    $k$ & $\|A_{k,k}\|$ & $\|A_{k,k}^{-1}\|$ & $p$  & $\|A_{k,\eta^*}\|$ & $\|A_{k,\eta^*}^{-1}\|$ & $p$ \\
  \hline
    $10^{-5}$ & 1.978 &  3087.76 &   &1.978 &   25.165 &   \\
    $10^{-4}$&  1.978 &   924.25&   -0.53 &1.978&  21.533 & -0.07 \\
    $10^{-3}$&  1.978 &   128.19&   -0.86   &1.978   &  17.581 & -0.09  \\
    $10^{-2}$&  1.978 &  45.28&      -0.45  &1.978  & 14.078  & -0.10\\
    $10^{-1}$&   1.978 &   14.43   &   -0.50 & 1.978 & 10.384 & -0.13\\
    $1$ &    2.050        &   4.53 &     -0.50 &  3.310&  3.570 & -0.46\\
   \hline
  \end{tabular}
  \caption{{\bf Rectangle of side lengths 2 and 0.02.} Galerkin BEM approximations to $\|A_{k,k}\|$, $\|A_{k,k}^{-1}\|$, $\|A_{k,\eta^*}\|$ , and $\|A_{k,\eta^*}^{-1}\|$.}

  \label{table:recteps002_3}
\end{center}
\end{table}

\subsection{Rectangular cavity} In the last two numerical examples we explore trapping domains, as studied theoretically in \S\ref{sec:trapping}.
The rectangular cavity in Figure \ref{fig:shapes1} is defined by the polygon with the following coordinates: $p_0=0$, $p_1=(-c,0)$, $p_2=(-c, -\ell)$, $p_3=(\ell,-\ell)$, $p_4=(\ell,2c-\ell)$, $p_5=(-c,2c-\ell)$, $p_6=(-c,2a)$, $p_7=(0,2a)$. Here, $a=\pi/10$, $c=1$, and $\ell=c-a$. The width of the cavity is $2a=\pi/5$. Hence we expect resonance values in the negative half of the complex plane close to $k=5m$, $n\in\N$ (see \cite[Figure 5.12]{BeSp:10} for numerical computations of exterior resonances for this cavity). Since for $k=5m$ the width of the cavity is an integer multiple of half a wavelength, Theorem~\ref{th:trapping} applies and implies that, for $k=5m$ and $m\in \N$,
$$
\|A^{-1}_{k,\eta}\| \geq C k^{9/10},
$$
for both $\eta = k$ and $\eta = k^{2/3}$. Hence, as $k\rightarrow\infty$, $\|A^{-1}_{k,k}\|$ will grow, a behaviour not observed in the previous examples.

Estimates
for $\|A_{k,k}\|$, $\|A_{k,k}^{-1}\|$,
$\|A_{k,k^{2/3}}\|$ and $\|A_{k,k^{2/3}}^{-1}\|$ for various $k$ are
shown in Table~\ref{table:trap4} below.
As expected, the value of $\|A_{k,k}^{-1}\|$ grows at  a rate at least as fast as $k^{9/10}$, and $\cond A_{k,k}$ grows at approximately the rate $k^{7/5}$ predicted by Theorem \ref{th:condtr1}. The growth of $\|A_{k,k^{2/3}}^{-1}\|$ and $\cond A_{k,k^{2/3}}$ are more erratic. Theorem \ref{th:6} implies that $\|D_k\|\geq Ck^{1/4}$, and $\|S_k\|\leq C k^{1/2}$ by \eqref{eq:Sk2Dupper}, so that, for this obstacle, $\|A_{k,k^{2/3}}\|\geq C k^{1/4}$. Thus the theoretical lower bound for  $\cond A_{k,k^{2/3}}$ is  $\cond A_{k,k^{2/3}} \geq C k^{23/20}$, a slower rate of growth than that proved for $\cond A_{k,k}$. Nevertheless, $\cond A_{k,k}< \cond A_{k,k^{2/3}}$ for the larger values of $k$ in Table \ref{table:trap4}.

\begin{table}[htbp]
\begin{center}
  \begin{tabular}{|r|rr|rr|rr|rr|}
  \hline
    $k$ & $\|A_{k,k}\|$ & $p$ & $\|A_{k,k}^{-1}\|$ & $p$ & $\|A_{k,k^{2/3}}\|$ & $p$ & $\|A_{k,k^{2/3}}^{-1}\|$ & $p$ \\
  \hline
    5 &  4.835 &      &         1.969 &      & 3.575 &    &  3.075  &    \\
   10 &  5.201 & 0.11 &    3.121 & 0.66 & 3.580 &   0.00 &  5.990 &   0.96 \\
   20 &  5.629 & 0.11 &    5.539 & 0.83 & 3.594 &   0.01 &  9.348&   0.64 \\
   40 &  6.182 & 0.14 &    10.322 & 0.90 &  3.790 &   0.08 &  19.029 &   1.03 \\
   80 &  8.112 & 0.39 &    19.774 & 0.94 &  4.223 &   0.16 &  28.528 &   0.58 \\
  160 & 11.066 & 0.45 &   38.351  & 0.96 &  4.788 &   0.18 &  173.563 &   2.60 \\
  320 & 15.254 & 0.46 &   75.156 & 0.97 &  5.483 &    0.20    & 277.480 &     0.68   \\
  \hline
  \end{tabular}
  \caption{{\bf Rectangular cavity.} Galerkin BEM approximations to $\|A_{k,k}\|$, $\|A_{k,k}^{-1}\|$, $\|A_{k,k^{2/3}}\|$ and $\|A_{k,k^{2/3}}^{-1}\|$.}
  \label{table:trap4}
\end{center}
\end{table}

\subsection{Elliptic cavity} The elliptic cavity in Figure \ref{fig:shapes1} is defined by two elliptic arcs. The first one is parameterised as $(\cos t,0.5\sin t)$, $t\in[-\phi_0,\phi_0]$, with $\phi_0=7\pi/10$ and the second arc is defined by $(1.3\cos t,0.6\sin t)$, $t\in[-\phi_1,\phi_1]$, with $\phi_1=\arccos(\frac{1}{1.3}\cos\phi_0)$. As discussed in \S\ref{sec:trapping}, we expect large values of $\|A_{k,k}^{-1}\|$ at $k$ values corresponding to so-called `bouncing ball' eigenmodes of the inner ellipse, which has semi-axes 1 and 0.5. Some of these modes, discussed in detail in the appendix, are shown in Figure \ref{fig:ellipsemodes}. As $k$ grows, these modes become more localized around the centre of the ellipse, i.e.\ around the stable periodic orbit (see \S\ref{sec:trapping} and the appendix).

Estimates for  $\|A_{k,k}\|$ and $\|A_{k,k}^{-1}\|$ for the four wavenumbers from Figure \ref{fig:ellipsemodes} are shown in Table \ref{table:ellipsecavity2}. We expect, from Theorems \ref{th:trapping2} and \ref{th:trapcond}, exponential growth of $\|A_{k,k}^{-1}\|$ and $\cond A_{k,k}$ with $k$. Comparing the values of $\|A_{k,k}^{-1}\|$ for the two lowest wavenumbers a large growth is clearly visible. However, the growth of $\|A_{k,k}^{-1}\|$ then levels off. It may be that the given discretisation with  about $10$ elements per wavelength is not sufficient
to resolve the large norm of the inverse, or that, due to discretisation error, the resonance wavenumber $k$ is shifted somewhat from the theoretically predicted value.
Repeating the computation with  approximately $20$ elements per wavelength,
we see the computed values for $\|A_{k,k}^{-1}\|$ for larger $k$ are significantly larger now.
However, the growth of the inverse still levels off, and we do not observe exponential growth even with this finer discretisation. We note, however, that with this finer discretisation we see, at the highest value of $k$, a condition number $\cond A_{k,k}\approx 140,000$, hugely larger than the values observed in any of the previous examples, including the rectangular cavity.

\begin{figure}
\center
\includegraphics[width=12cm]{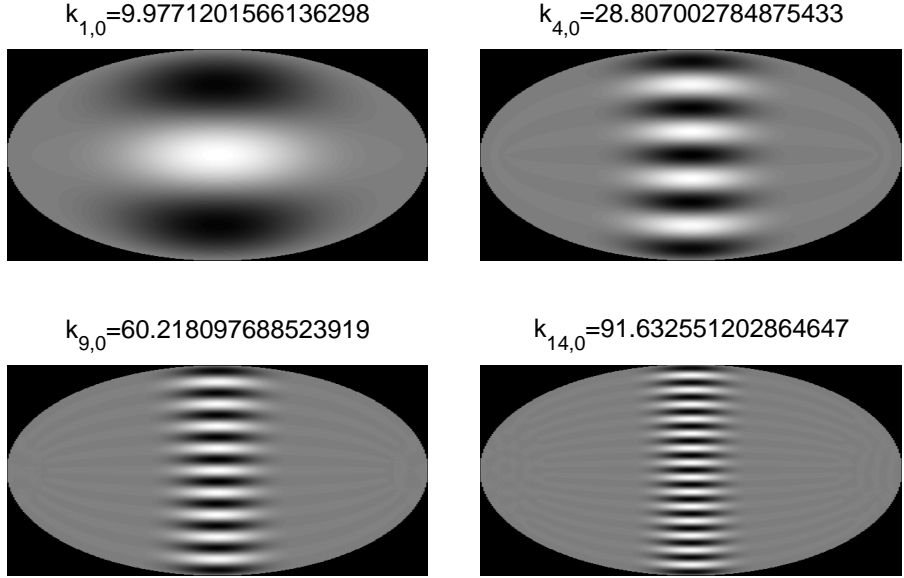}
\caption{Some exponentially localised modes of the ellipse associated with zeros of the radial Mathieu function of order $0$. In the notation of the appendix, the modes plotted are $u_{m,0}$ for $m=1,4,9,$ and 14, corresponding to the  wavenumbers $k_{m,0}$, $m=1,4,9,14$.}
\label{fig:ellipsemodes}
\end{figure}

\begin{table}[htbp]
\begin{center}
  \begin{tabular}{|r|rr|r|r|}
  \hline
    $k$ & $\|A_{k,k}\|$ & $p$ & $\|A_{k,k}^{-1}\|$ & $\|A_{k,k}^{-1}\|$ (refined) \\
\hline
    9.977 & 5.595 & & 67.3  & 70.1 \\
   28.807 &  7.294 & 0.25 & 2829.9 & 9182.7\\
   60.218 &  8.820 & 0.26 & 5265.0 & 13373.8 \\
   91.633 & 10.144 & 0.33 & 6543.8 & 14258.6\\
   \hline
  \end{tabular}
  \caption{{\bf Elliptic cavity.} Galerkin BEM approximations to $\|A_{k,k}\|$ and $\|A_{k,k}^{-1}\|$. For $\|A_{k,k}^{-1}\|$ the results of a refined computation with twice the number of degrees of freedom and approximately $20$ elements per wavelength are shown.}
  \label{table:ellipsecavity2}
\end{center}
\end{table}

\section{Conclusions}
\label{sec:conc}
\setcounter{equation}{0}
In this paper we have, in \S\ref{sec:previous}, summarised what is known regarding upper and lower bounds on the norms of the acoustic single- and double-layer potential operators, $S_k$ and $D_k$, and the combined layer potential operator $A_{k,\eta}$, with an emphasis on how these bounds behave as a function of frequency, and the influence of the shape of the boundary. We have also proved sharper upper bounds on $\|S_k\|$ and $\|A_{k,\eta}\|$ for low $k$, have summarised what upper and lower bounds on $\|A_{k,\eta}^{-1}\|$ are known, and have shown that exponential growth of $\|A_{k,\eta}^{-1}\|$ is possible as $k\to\infty$ through some sequence of wave numbers, in the case of a certain class of 2D trapping obstacles. Finally, we have discussed the condition number $\cond A_{k,\eta}$, proving that it remains bounded as $k\to0$ with appropriate choices of the coupling parameter $\eta$, and showing that, while it increases as $k\to\infty$ only as fast as $k^{1/3}$ for a circle or sphere, and at the rate $k^{1/2}$ for a starlike polygon, it grows exponentially, as $k$ increases through some sequence, for certain trapping obstacles.

In \S\ref{sec:discrete} we have explored the implications of these results for Galerkin BEM discretisations of these operators, showing that the norms of the Galerkin BEM matrices converge to the norms of the operators that they discretise, as the mesh is refined, and provided an orthonormal basis is used. Convergence to $\|A_{k,\eta}^{-1}\|$ of the norm of the inverse of the matrix corresponding to $A_{k,\eta}$ has also been proved in the case that $\Gamma$ is $C^1$. Thus we expect that the norm bounds at the continuous level in \S2 will apply also at the discrete level if the mesh is sufficiently refined.

This has been confirmed in \S\ref{sec:num} where we have explored a range of numerical examples, including shapes that are convex (both smooth and non-smooth), non-convex but starlike, and non-starlike trapping obstacles. The quantitative upper and lower bounds stated in \S\ref{sec:previous} are found to be upper and lower bounds also at the discrete level, and to be rather sharp in many of the examples. In almost all cases the observed rate of growth of norms and condition numbers as $k$ increases is in accordance with the possible range of behaviour suggested by the upper and lower bounds from \S\ref{sec:previous}, with the rates of growth mainly closer to the lower bounds of \S\ref{sec:lower}. The exception is that the exponential growth predicted in Theorems \ref{th:trapping2} and \ref{th:trapcond} is not observed numerically in the `Elliptic cavity' example, at least at the discretisations we use. On the other hand, the condition numbers observed in this case are as high as 140,000 and the values of $\|A_{k,k}^{-1}\|$ as high as 14,000, which contrasts with $\|A_{k,k}^{-1}\|\approx 1$ in all the cases where the scatterer $\Omega$ is starlike.

\vspace{2ex}

\renewcommand{\theequation}{A.\arabic{equation}}
\setcounter{equation}{0}
\renewcommand{\thetheorem}{A.\arabic{theorem}}
\setcounter{theorem}{0}
{\bf Appendix: Eigenmodes of the Ellipse.} In this appendix we summarise key properties of eigenmodes of the Laplacian in an elliptical domain.
Suppose $a_1>a_2>0$ and let $E = \{(x_1,x_2)\in \R^2 : (x_1/a_1)^2+(x_2/a_2)^2 \leq 1\}$ be the ellipse with semi-major axis $a_1$ and semi-minor axis $a_2$. Then we study in this appendix certain Laplace eigenmodes $u\in C^2(\bar E)$ satisfying
\begin{equation} \label{eigen}
\Delta u + k^2 u = 0 \mbox{ in }E, \quad u = 0 \mbox{ on } \partial E,
\end{equation}
for some $k>0$.

 Let $a=\sqrt{a_1^2-a_2^2}= a_1 \varepsilon$, where $\varepsilon = \sqrt{1-a_2^2/a_1^2}$ is the eccentricity of the ellipse, and introduce elliptical coordinates $(\mu,\nu)$, defined by
 $$
 x_1=a\cosh \mu \cos \nu \;\;\mbox{ and }\;\; x_2 = a \sinh \mu \sin \nu,
 $$
  in terms of which
 $$
 E =  \{(a\cosh \mu \cos \nu,a \sinh \mu \sin \nu):0\leq \mu \leq \mu_0, \; 0\leq\nu<2\pi\},
 $$
 where $\mu_0 := \tanh^{-1}(a_2/a_1)$. It is well known (see e.g.\ \cite{wilson}) that the Laplace operator separates in elliptical coordinates, and that in this coordinate system the Helmholtz equation can be written as
 \begin{equation} \label{he_ell}
 \left(\frac{\partial^2}{\partial \mu^2} + \frac{\partial^2}{\partial \nu^2} + k^2 a^2(\sinh^2\mu+\sin^2\nu)\right)u = 0.
 \end{equation}
 Seeking separation of variables solutions in the form $u(x)=M(\mu)N(\nu)$, we see that \eqref{he_ell} implies that $N$ satisfies the circumferential (or standard) Mathieu equation
 \begin{equation} \label{mat_st}
N^{\prime\prime}(\nu) + (\alpha-2q\cos 2\nu)N(\nu) = 0,
 \end{equation}
 while $M$ satisfies the radial (or modified) Mathieu equation
 \begin{equation} \label{mat_mo}
M^{\prime\prime}(\mu) - (\alpha-2q\cosh 2\mu)M(\mu) = 0.
 \end{equation}
 In these equations
 $$
 q = \frac{1}{4}(ka)^2 = \frac{1}{4}(ka_1)^2(1-(a_2/a_1)^2)
 $$
 and $\alpha$ is a separation constant. The solutions to \eqref{mat_st} that we are interested in are the solutions of period $2\pi$, satisfying $N(0)=N(2\pi)$ and $N^\prime(0)=N^\prime(2\pi)$. Since $N(-\nu)$ is such a solution if $N(\nu)$ is, it is clear that we may restrict attention to  periodic solutions of \eqref{mat_st} that are either even or odd.

 For this paper it is enough to focus on the even periodic solutions, so that we seek solutions $N$ of \eqref{mat_st} which satisfy $N^\prime(0)=N^\prime(\pi)$. This is an eigenvalue problem in which the eigenvalue is the separation constant $\alpha$. Standard Sturm-Liouville theory tells us that, for each value of the parameter $q>0$, there are a countable number of eigenvalues $\alpha = a_n(q)$, $n=0,1,\dots$, with $a_0(q)<a_1(q)<\dots$ and $a_n(q)\to\infty$ as $n\to\infty$ (we use in this appendix the standard notation for these eigenvalues and the corresponding eigenfunctions, see e.g.\  \cite[\S28.2(v)]{dlmf}). Further, to the eigenvalue $a_n(q)$ there corresponds a unique (to within multiplication by a constant), real-valued eigenfunction $N(\nu)$, which is usually denoted by $\mathrm{ce}_n(\nu,q)$ (for the standard normalization of $\mathrm{ce}_n$ see \cite[\S28.2(vi)]{dlmf}). The standard Sturm-Liouville theory tells us that $\mathrm{ce}_n(\nu,q)$ has precisely $n$ zeros in $(0,\pi)$. It is easy to see that $\mathrm{ce}_n(\nu,q)\pm\mathrm{ce}_n(\pi-\nu,q)$ is also an eigenfunction corresponding to $\alpha = a_n(q)$; thus the uniqueness of the normalised eigenfunction implies that, for $\nu\in \R$ and $n\geq 0$,
 \begin{equation} \label{even}
\mathrm{ce}_n(\nu,q) = \mathrm{ce}_n(\pi-\nu,q) \mbox{ if $n$ is even}, \quad \mathrm{ce}_n(\nu,q) = -\mathrm{ce}_n(\pi-\nu,q) \mbox{ if $n$ is odd}.
 \end{equation}

 Given that $\alpha= a_n(q)$, for some $n\geq0$, and that $N = \mathrm{ce}_n(\cdot, q)$, it is a standard result (e.g.\ \cite{mclachlan}) that $u(x)=M(\mu)N(\nu)$ satisfies the Helmholtz equation \eqref{eigen} in $E$ if and only if $M\in C^2(\R)$ is an even function that satisfies \eqref{mat_mo}. This uniquely specifies $M$ to within multiplication by a constant. The standard notation for this (real-valued) solution is  $M(\mu) = \mathrm{Mc}_n^{(1)}(\mu,q)$; see \cite[\S28.20(iv)]{dlmf} for the standard normalization. Thus we see that $u(x)=M(\mu)N(\nu)=\mathrm{Mc}_n^{(1)}(\mu,q)\mathrm{ce}_n(\nu, q)$ satisfies the full eigenvalue problem \eqref{eigen} if and only if
 \begin{equation} \label{mode}
 \mathrm{Mc}_n^{(1)}(\mu_0,q)=0.
 \end{equation}

 The complication in computing eigenmodes of the ellipse (for methods see \cite{wilson,neves10}) is that it is a multi-parameter spectral problem: to satisfy \eqref{mode} we have to find a pair $(\alpha,q)$ such that, simultaneously, \eqref{mat_st} has a periodic solution and \eqref{mat_mo} has a solution which is even if $N$ is even and which vanishes at $\mu_0$. Neves \cite{neves10} gives a proof based on multi-parameter spectral theory that for each pair $(m,n)\in \{0,1,\ldots\}^2$ there exists a unique $q_{m,n}>0$ such that \eqref{mode} holds with $\mathrm{Mc}_n^{(1)}(\cdot,q_{m,n})$ having $m$ zeros in $(0,\mu_0)$. The function
 \begin{equation} \label{eigenf}
 u(x)=u_{m,n}(x) := \mathrm{Mc}_n^{(1)}(\mu,q_{m,n})\mathrm{ce}_n(\nu, q_{m,n})
 \end{equation}
 is then an eigenfunction of \eqref{eigen} for $k=k_{m,n}:=\sqrt{4q_{m,n}}/a$. It is well known (e.g.\ \cite{CK83}) that the eigenvalues of the Laplace operator have infinity as the only accumulation point, so that $k_{m,n}\to \infty$ as $m+n\to\infty$.

 For some $\nu_0\in (0,\pi/2)$ let
\begin{eqnarray*}
 E_{\nu_0} &:=& \{(a\cosh \mu \cos \nu,a \sinh \mu \sin \nu):0\leq \mu <\mu_0, \; |\nu|< \nu_0 \mbox{ or }|\pi-\nu|<\nu_0\}\\
 &\supset & \{(x_1,x_2)\in E:|x_1|>a_1\cos \nu_0\}.
\end{eqnarray*}
 Let
 $$
 \rho_{\nu_0}(m,n) := \left\{\frac{\int_{E_\nu} (u_{m,n})^2\,dx}{\int_{E} (u_{m,n})^2\,dx} \right\}^{1/2}.
 $$
 Our particular interest in this appendix is in families of eigenfunctions that are exponentially localised around the periodic orbit $\{(0,x_2):|x_2|\leq a_2\}$. In particular we will show below that the family $u_{m,0}$, $m= 0,1,\ldots$ is so localised; precisely, we will show that, for all $\nu_0\in (0,\pi/2)$, there exists $\beta>0$ such that $\rho_{\nu_0}(m,0)= O(\re^{-\beta k_m})$ as $m\to\infty$.

 Noting \eqref{even}, we see that
 \begin{eqnarray*}
 \int_{E_\nu} (u_{m,n})^2\,dx & = & 4 a^2\int_0^{\nu_0}\int_0^{\mu_0} (\sinh^2\mu + \sin^2\nu)\left(\mathrm{Mc}_n^{(1)}(\mu,q_{m,n})\mathrm{ce}_n(\nu, q_{m,n})\right)^2d\mu d\nu.
 \end{eqnarray*}
 Thus, defining
 $$
 M_j := \int_0^{\mu_0} (\sinh\mu)^{2j}\left(\mathrm{Mc}_n^{(1)}(\mu,q_{m,n})\right)^2d\mu, \quad I_s(m,n) := \int_{0}^{s} \left(\mathrm{ce}_n(\nu, q_{m,n})\right)^2 d\nu,
 $$
 it holds that
 \begin{eqnarray} \nonumber
 \left(\rho_{\nu_0}(m,n)\right)^2 &\leq &\frac{ (M_1  + M_0\sin^2\nu_0) I_{\nu_0}(m,n)}{M_1 I_{\pi/2}(m,n) +  M_0\sin^2\nu_0(I_{\pi/2}(m,n) - I_{\nu_0}(m,n))}\\
 \label{rhobound}
 & \leq & \frac{ I_{\nu_0}(m,n)}{I_{\pi/2}(m,n) - I_{\nu_0}(m,n)}.
 \end{eqnarray}

 It is sufficient for the needs of this paper to estimate the asymptotics as $m\to\infty$ of $\rho_{\nu_0}(m,n)$ for $n=0$, and so we will restrict our attention to this case. For this purpose, and abbreviating $q_{m,0}$ as $q_m$ and $k_{m,0}$ as $k_m$, recall that $\mathrm{ce}_0(\nu, q_{m})$ satisfies \eqref{mat_st} with $q = q_{m}$ and with $\alpha = a_0(q_{m})$. Now the asymptotics of the eigenvalue $a_0(q)$ as $q\to\infty$ are known. From \cite{dlmf} we have that
 \begin{equation} \label{a0q}
 a_0(q) = -2q + q^{1/2} + O(1)
 \end{equation}
 as $q\to\infty$.
 Thus we see that, for $m$ large and with $N= \mathrm{ce}_0(\cdot, q_{m})$, the coefficient $a_0(q_m)-2q_m\cos 2\nu$ of $N(\nu)$ in \eqref{mat_st} is negative except in small neighbourhoods of $\pm \pi/2$ of length $O(q_m^{-1/4}) = O(k_m^{-1/2})$. It is this which causes the exponential localisation of $u_{m,0}$ around the periodic orbit.

 To see this localization completely explicitly, we will use the following lemma which depends on standard weighted space arguments (cf.\ \cite{agmon82}, \cite[\S3]{HislopSegal}). In this lemma and subsequently $BC(\R)$ denotes the set of functions $\phi:\R\to\R$ that are bounded and continuous and $BC^k(\R):= \{\phi\in BC(\R):\phi^{(j)}\in BC(\R) \mbox{ for } j = 0,1,\ldots,k\}$. As usual, $H^s(\R)$, for $s\geq 0$, denotes the standard Sobolev space of order $s$ (which in this appendix we take to be a space of real-valued functions). For $\beta>0$ and $\alpha\geq 0$ let $w_{\beta,\alpha}:\R\to\R$ denote the weight function
 $$
 w_{\beta,\alpha}(s) = \exp\left(-\beta \sqrt{\alpha^2+s^2}\, \right),
 $$
and, for $\phi\in BC(\R)$, let
 $$
 \|\phi\|_{\beta,\alpha} := \left\{\int_{-\infty}^\infty (w_{\beta,\alpha}(s)\phi(s))^2 ds \right\}^{1/2} < \infty.
 $$

 \begin{lemma} \label{weighted}
 Suppose that $p\in BC(\R)$ and that, for some $c>0$, $p(s) \leq -c^2 < 0$, $s\in \R$. Suppose also that $g\in BC(\R)$, $v\in BC^2(\R)$, and
 $$
 v^{\prime\prime}(s) + p(s) v(s) = g(s), \quad s\in \R.
 $$
 Then, for $0<\beta<c$ and $\alpha\geq 0$,
 \begin{equation} \label{bbound}
 \|v\|_{\beta,\alpha} \leq \frac{1}{c^2-\beta^2}\,\|g\|_{\beta,\alpha}.
\end{equation}
\end{lemma}
\begin{proof}
Suppose $\beta >0$ and $\alpha>0$. Define $\psi\in BC^2(\R)\cap H^2(\R)$ by $\psi(s) = w_{\beta,a}(s)v(s)$, $s\in \R$. Then it is an easy calculation, abbreviating $w_{\beta,\alpha}$ as $w$, that
$$
\psi^{\prime\prime} -2\frac{w^\prime}{w} \psi^\prime + \left(p+2\left(\frac{w^\prime}{w}\right)^2 - \frac{w^{\prime\prime}}{w}\right) \psi = wg.
$$
Multiplying by a test function $\phi$ and integrating by parts, we see that
\begin{equation} \label{wf}
a(\psi, \phi) = b(\phi), \quad \phi\in H^1(\R),
\end{equation}
where the bilinear form $a: H^1(\R)\times H^1(\R)\to \R$ and the bounded linear functional $b:H^1(\R)\to \R$ are defined by
\begin{eqnarray*}
a(\phi,\psi) &:=& \int_{-\infty}^\infty \left(\psi^\prime\phi^\prime + 2\frac{w^\prime}{w} \psi^\prime \phi - \left(p+2\left(\frac{w^\prime}{w}\right)^2 - \frac{w^{\prime\prime}}{w}\right) \psi\phi\right) ds,\\ b(\phi) &:=& -\int_{-\infty}^\infty wg\phi\, ds.
\end{eqnarray*}
 Since $w^\prime/w$, $w^{\prime\prime}/w\in BC(\R)$, the bilinear form $a$ is bounded. For $\phi\in H^1(\R)$,
\begin{eqnarray*}
a(\phi,\phi) &=&   \int_{-\infty}^\infty \left((\phi^\prime)^2 + \frac{w^\prime}{w} (\phi^2)^\prime - \left(p+2\left(\frac{w^\prime}{w}\right)^2 - \frac{w^{\prime\prime}}{w}\right) \phi^2\right) ds\\
&  = &  \int_{-\infty}^\infty \left((\phi^\prime)^2 - \left(p+\left(\frac{w^\prime}{w}\right)^2\right) \phi^2\right) ds,
\end{eqnarray*}
where the last step follows by integration by parts, on noting that $(w^\prime/w)^\prime = w^{\prime\prime}/w - (w^\prime/w)^2$.  Since $-p-(w^\prime/w)^2\geq c^2-\beta^2$, $a$ is coercive if $\beta <c$, with
$$
a(\phi,\phi) \geq  \|\phi\|_1^2, \mbox{ where } \|\phi\|_1 :=  \left(\int_{-\infty}^\infty \left((\phi^\prime)^2 + \left(c^2-\beta^2\right)\phi^2\right) ds\right)^{1/2}.
$$
Applying the Lax-Milgram lemma, it follows from \eqref{wf} that
$$
\sqrt{c^2-\beta^2}\|v\|_{\beta,\alpha} \leq \|\psi\|_1 \leq \|b\| \leq \frac{\|g\|_{\beta,\alpha}}{\sqrt{c^2-\beta^2}},
$$
where $\|b\|$ denotes the norm of the linear functional $b:H^1(\R)\to \R$, with $H^1(\R)$ given the norm $\|\cdot\|_1$. Hence \eqref{bbound} holds for $\beta>0$ and $\alpha>0$, and so also for $\alpha=0$ by the dominated convergence theorem.
\end{proof}

To apply this lemma to \eqref{mat_st}, let $v:= \mathrm{ce}_0(\cdot, q)$ and define $p\in BC(\R)$ by $p(\nu)=a_0(q)-2q\cos 2\nu$, $\nu\in\R$. For $c>0$ write $p$ as $p=p_c^-+p_c^+$ where $p_c^- := \min(-c^2, p)$ and $p_c^+ := p - p_c^-$, and set $g_c:=-p_c^+v$. Then
$
v^{\prime\prime} + p_c^- v = g_c
$,
and applying Lemma \ref{weighted} with $\alpha=0$ it follows that, for $\nu_0\in (0,\pi/2)$,
\[
\left\{2\int_0^{\nu_0} v^2 ds\right\}^{1/2} \leq \re^{\beta \nu_0} \|v\|_{\beta,0} \leq \frac{\re^{\beta \nu_0}}{c^2-\beta^2}\,\|g_c\|_{\beta,0},
\]
for $0<\beta<c$. For $-2q<a_0(q)+c^2<2q$ we see that $g_c(\nu) = 0$ if $|\nu-j\pi| \leq \nu_c$, for some $j\in\Z$, where $\nu_c\in (0,\pi/2)$ is given by
\begin{equation} \label{nuc}
\nu_c := \frac{1}{2} \cos^{-1}\left(\frac{a_0(q)+c^2}{2q}\right).
\end{equation}
Thus, and since $0\leq p_c^+ \leq a_0(q)+2q+c^2< 4q$, it follows that
\begin{eqnarray*}
\|g_c\|^2_{\beta,0} &=& 2 \int_0^\infty \re^{-2\beta s} (g_c(s))^2\, ds\\
& \leq & 8q\sum_{j=0}^\infty \int_{j\pi+\nu_c}^{(j+1)\pi-\nu_c}\re^{-2\beta s} (v(s))^2\, ds\\
& \leq & 8q \frac{\re^{-2\beta \nu_c}}{1-\re^{-2\beta \pi}} \int_0^\pi v^2\, ds.
\end{eqnarray*}
So
$$
\left\{\int_0^{\nu_0} v^2 ds\right\}^{1/2} \leq 2 \sqrt{2q}\, \frac{\re^{-\beta(\nu_c- \nu_0)}}{(c^2-\beta^2)\left(1-\re^{-2\beta \pi}\right)^{1/2}}\left\{\int_0^{\pi/2} v^2 ds\right\}^{1/2}
$$
and choosing $\beta = c^2(\nu_c-\nu_0)/(1+\sqrt{1+(\nu_c-\nu_0)^2c^2\,}\,)\in (0,c)$, which minimises $\re^{-\beta(\nu_c- \nu_0)}/(c^2-\beta^2)$, we find that
\begin{equation} \label{mainbound}
\left\{\int_0^{\nu_0} v^2 ds\right\}^{1/2} \leq  \frac{(\nu_c-\nu_0)^2 \sqrt{2q}\,\re^{-\gamma}}{\gamma\left(1-\re^{-4\gamma \pi}\right)^{1/2}}\left\{\int_0^{\pi/2} v^2 ds\right\}^{1/2},
\end{equation}
where $\gamma := \delta/(1+\sqrt{1+\delta})$, with $\delta := c^2(\nu_c-\nu_0)^2$, provided that $-2q<a_0(q)+c^2<2q$ and $0<\nu_0<\nu_c$. Thus, assuming that $-a_0(q)/(2q)\in (-1,0)$ (which is certainly the case for all sufficiently large $q$ by \eqref{a0q}), and provided
\begin{equation} \label{nurange}
0<\nu_0 < \frac{1}{2}\cos^{-1}\left(\frac{a_0(q)}{2q}\right),
\end{equation}
\eqref{mainbound} holds for $0< c< \sqrt{2q\cos 2\nu_0 - a_0(q)\,}$. It particular, choosing
\begin{equation} \label{cform}
c := \frac{1}{2}\sqrt{2q\cos 2\nu_0 - a_0(q)\,},
\end{equation}
we obtain the following result in which the formula for $\delta$ follows from $\delta = c^2(\nu_c-\nu_0)^2$, with $c$ given by \eqref{cform} and $\nu_c$ by \eqref{nuc}, recalling that $\cos^{-1}(-1+\alpha)=\pi-2\sin^{-1}\sqrt{\alpha/2}$, for $0\leq\alpha\leq 2$.
\begin{theorem} \label{mainthmapp}
Let $v:= \mathrm{ce}_0(\cdot, q)$ and suppose that $-a_0(q)/(2q)\in (-1,0)$ (which certainly holds for all sufficiently large $q$). Then, provided $\nu_0$ satisfies \eqref{nurange}, it holds that
\[
\left\{\int_0^{\nu_0} v^2 ds\right\}^{1/2} \leq  \frac{\phi_0^2 \sqrt{2q}\,\re^{-\gamma}}{\gamma\left(1-\re^{-4\gamma \pi}\right)^{1/2}}\left\{\int_0^{\pi/2} v^2 ds\right\}^{1/2},
\]
where $\phi_0 = \pi/2-\nu_0$, $\gamma = \delta/(1+\sqrt{1+\delta})$,
$$
\delta = q(\sin^2\phi_0-r(q))\left(\phi_0 - \sin^{-1}\left(\frac{1}{2}\sqrt{\sin^2\phi_0+3r(q)\,}\,\right)\right)^2.
$$
and $r(q):= (1+a_0(q)/(2q))/2$.
\end{theorem}

By \eqref{a0q}, $r(q) \sim q^{-1/2}/2$ as $q\to\infty$. Thus, in the above theorem, as $q\to\infty$ $\delta$ has the asymptotic behaviour $\delta \sim q\sin^2\phi_0\left(\phi_0 - \sin^{-1}\left(\frac{1}{2}\sin\phi_0\right)\right)^2$. Thus, and since $\sin^{-1}\left(\frac{1}{2}\sin\phi_0\right)< \frac{1}{2}\phi_0$, it holds that $\delta > \frac{1}{4}q\phi_0^2\sin^2\phi_0$ for all sufficiently large $q$. Further,  $\gamma =  \delta^{1/2} - 1 +O(\delta^{-1/2})$ as $\delta\to\infty$. Thus the above theorem has the following corollary.

\begin{corollary} \label{cor_final}
Let $v:= \mathrm{ce}_0(\cdot, q)$. Then, for every $\nu_0\in (0,\pi/2)$, it holds for all sufficiently large $q$ that
\[
\left\{\int_0^{\nu_0} v^2 ds\right\}^{1/2} \leq  \frac{2\sqrt{2}\,\re\,\phi_0}{\sin\phi_0}\,\exp\left(-\frac{1}{2}\sqrt{q}\phi_0\sin \phi_0\right) \left\{\int_0^{\pi/2} v^2 ds\right\}^{1/2},
\]
where $\phi_0 = \pi/2-\nu_0$.
\end{corollary}

Applying this corollary with $q=q_m = \frac{1}{4}(k_ma)^2$, and recalling the bound \eqref{rhobound}, and that $a=a_1\varepsilon$, where $\varepsilon$ is the eccentricity of the ellipse, and that $\phi_0/\sin \phi_0 < \pi/2$, we see that, for every $\nu_0\in (0,\pi/2)$, it holds for all sufficiently large $m$ that
\begin{equation} \label{finalfinal}
\rho_{\nu_0}(m,0) \leq \sqrt{2}\,\pi \re \,\exp\left(-\frac{1}{4}\varepsilon k_m a_1 \phi_0 \sin \phi_0\right),
\end{equation}
where $\phi_0 = \pi/2-\nu_0$.

\begin{thebibliography}{00}

\bibitem{AbSt:72}
M.~Abramowitz and I.~Stegun,
\newblock {Handbook of Mathematical Functions},
\newblock Dover, New York, 1972.

\bibitem{agmon82}
S.~Agmon,
\newblock {Lectures on Exponential Decay of Solutions of Second-Order Elliptic Equations},
\newblock Princeton University Press, Princeton NJ, 1982.

\bibitem{Ami90}
S.~Amini,
\newblock On the choice of the coupling parameter in boundary integral equation formulations of the exterior acoustic
problem,
\newblock { Appl. Anal.}, {\bf 35}, (1990), 75--92.

\bibitem{Ami93}
S.~Amini,
\newblock Boundary integral solution of the exterior acoustic problem,
\newblock {Comput. Mech.}, {\bf 13}, (1993), 2--11.

\bibitem{BaSa:07}
L.~Banjai and S.~Sauter,
\newblock A refined Galerkin error and stability analysis for highly indefinite variational problems,
\newblock {SIAM J. Numer. Anal.} {\bf 45}, (2007), 37--53.

\bibitem{BeSp:10}
T.~Betcke and E.A.~Spence,
\newblock Numerical estimation of coercivity constants for boundary integral operators in acoustic scattering,
\newblock submitted to {SIAM J. Numer. Anal.}


\bibitem{BrKu:01}
O.~P.~Bruno and A.~L.~Kunyansky,
\newblock Surface scattering in three dimensions: an accelerated high-order solver,
\newblock {Proc. R. Soc. Lond. A} {\bf 457}, (2001), 2921--2934.

\bibitem{BrPC07}
{ O.~P.~Bruno}, Private communication, 2007.



\bibitem{CWL06}
S.~N.~Chandler-Wilde and S.~Langdon,
\newblock A Galerkin boundary element method for high frequency scattering by convex polygons,
\newblock {SIAM J. Numer. Anal.} {\bf 45}, (2007), 610--640.

\bibitem{CWGLL07}
S.~N.~Chandler-Wilde, I.~G.~Graham, S.~Langdon and M.~Lindner,
\newblock Condition number estimates for combined potential boundary integral operators
in acoustic scattering,
\newblock {J. Int. Eqn. Appl.} {\bf 21}, (2009), 229--279.

\bibitem{CWMonk06}
S.~N.~Chandler-Wilde and P.~Monk,
\newblock Wave-number-explicit bounds in time-harmonic scattering,
\newblock {SIAM J. Math. Anal.} {\bf 39}, (2008), 1428--145.

\bibitem{CK83}
D.~L.~Colton and R.~Kress,
\newblock {Integral equation methods in scattering theory,}
\newblock John Wiley, New York, 1983.


\bibitem{CK92}
\leavevmode\vrule height 2pt depth -1.6pt width 23pt, { Inverse
Acoustic and
  Electromagnetic Scattering Theory}, Springer Verlag, 2nd Ed., 1992.

\bibitem{dlmf}
Digital Library of Mathematical Functions, Release date 2010-05-07, National Institute of Standards and Technology from \verb+http://dlmf.nist.gov/+.

\bibitem{DoGrSm:07}
V.~Dominguez, I.~G.~Graham, and V.~P.~Smyshlyaev,
\newblock A hybrid numerical-asymptotic boundary integral method for
high-frequency acoustic scattering,
\newblock {Numer. Math.} {\bf 106}, (2007) 471--510.


\bibitem{Fabes78}
E.~B.~Fabes, M.~Jodeit, and N.M.~Riviere,
\newblock Potential techniques for boundary value problems on $C^1$ domains,
\newblock {Acta Math.}, {\bf 141}, (1978) 165--186.



\bibitem{Gi:97}
K.~Giebermann,
\newblock {Schnelle Summationsverfahren zur
numerischen L\"osung von Integralgleichungen f\"ur Streuprobleme im
$\R^3$,}
\newblock PhD Thesis, Universit\"at Karlsruhe, Germany, 1997.


\bibitem{HislopSegal}
P.~D. Hislop and I.~M.~Segal, {  Introduction to spectral theory: with applications to Schr\"odinger operators,}
Springer, New York, 1995.

\bibitem{Keller85}
J.~B.~Keller,
\newblock Semiclassical mechanics,
\newblock {SIAM Rev.}, {\bf 27}, (1985), 485--504.

\bibitem{Kr:85}
R.~Kress,
\newblock Minimizing the condition number of boundary integral-operators in acoustic and electromagnetic scattering,
\newblock {Q. J. Mech. Appl. Math.}, {\bf 38}, (1985), 323--341.

\bibitem{KrSp:83}
R.~Kress and W.~T.~Spassov,
\newblock On the condition number of boundary integral operators for the exterior Dirichlet problem for
the Helmholtz equation,
\newblock {Numer. Math.}, {\bf 42}, (1983), 77--95.

\bibitem{Kress}
R.~Kress,
\newblock {Linear Integral Equations,}
\newblock Springer, New York, 2nd Edition, 1999.










\bibitem{Lazutkin99}
V.~F.~Lazutkin, ``Semiclassical asymptotics of eigenfunctions'', in {Encyclopedia of Mathematical Sciences, Volume 34, Partial Differential Equations V}, M.~V.~Fedoryuk (Editor), Springer, New York, 1999, p.~133.

\bibitem{LoMe10}
M.~L\"ohndorf and  J.~M.~Melenk, {Wavenumber-explicit hp-BEM for high
frequency scattering}, ASC Report 2/2010, Institute for
Analysis and Scientific Computing, Vienna University of Technology -
TU Wien, 2010

\bibitem{mclachlan}
N.~W.~McLachlan,
\newblock {Theory and Application of Mathieu Functions},
\newblock Oxford University Press, Oxford, 1947.

\bibitem{mclean}
W.~McLean,
\newblock {Strongly Elliptic Systems and Boundary Integral Equations},
\newblock Cambridge University Press, Cambrideg, 2000.

\bibitem{melenk}
J.M. Melenk, {Mapping properties of combined field Helmholtz
boundary integral operators}, ASC Report 1/2010, Institute for
Analysis and Scientific Computing, Vienna University of Technology -
TU Wien, 2010.




\bibitem{prsil}
S.~Pr\"ossdorf and B.~Silbermann,
\newblock {Numerical Analysis for Integral and Related Operator
Equations,}
\newblock Birkh\"auser, 1991.



\bibitem{neves10}
A.~G.~M.~Neves,
\newblock Eigenmodes and eigenfrequencies of vibrating elliptic membranes: a {K}lein oscillation theorem and numerical calculations,
\newblock {Comm. Pure Appl. Anal.}, {\bf 9}, (2010), 611--624.




\bibitem{toth}
J.~Toth,
\newblock Eigenfunction decay estimates in the quantum integrable case,
\newblock {Duke Math. J.}, {\bf 93}, (1998), 231--255.


\bibitem{verchota}
{ G.~Verchota}, {Layer potentials and regularity for the
Dirichlet problem for Laplace's equation in Lipschitz domains}, J.
Funct. Anal., {\bf 59}, (1984), 572--611.

\bibitem{WaChew04}
{ K.F.~Warnick and W.C.~Chew}, { Error analysis of the moment
method}, IEEE Ant. Prop. Mag. {\bf 46}, (2004), 38--53.

\bibitem{WaChew01}
{ K.F.~Warnick and W.C.~Chew}, {On the spectrum of the
electric field integral equation and the convergence of the moment
method}, Int. J. Numer. Meth. Engng. {\bf 51}, (2001), 31--56.

\bibitem{WaChew99}
{ K.F.~Warnick and W.C.~Chew}, {Convergence of moment-method
solutions of the electric field integral equation for a 2-D open
cavity}, Microwave Optical Tech. Letters {\bf 23}, (1999), 212--218.


\bibitem{wilson}
H.~B.~Wilson and R.W.~Scharstein,
\newblock Computing elliptic membrane high frequencies by Mathieu and Galerkin methods,
\newblock {J. Eng. Math.}, {\bf 57}, (2007), 41--55.


\end{thebibliography}
\end{document}